\documentclass[11pt]{amsart}
\usepackage{amsmath,amssymb,amsthm}
\usepackage{enumerate}

\usepackage[english]{babel}

\usepackage{pstricks}
\usepackage{graphicx}
\usepackage{curve2e}
\usepackage{tikz}

\makeatletter
\@namedef{subjclassname@2010}{%
  \textup{2010} Mathematics Subject Classification}
\makeatother

\usepackage{pgfplots}
\usepgfplotslibrary{polar}

\usepackage{float}

\usepackage[utf8]{inputenc}
\usepackage[english]{babel}
\usepackage[T1]{fontenc}
\usepackage{lmodern}
\usepackage{tikz}
\usetikzlibrary{calc}
\usepackage{verbatim}

\usepackage{pstricks}
\usepackage{pgf}
\usepackage{curve2e}

\theoremstyle{plain}

\newtheorem{main-theorem}{Theorem}
\newtheorem{theo}{Theorem}[section]
\newtheorem{prop}[theo]{Proposition}
\newtheorem{lemm}[theo]{Lemma}

\newtheorem{deff}[theo]{Definition}
\newtheorem{exam}[theo]{Example}
\newtheorem{rema}[theo]{Remark}

\newtheorem{Thm}{Theorem}

\newtheorem{big-ques}[Thm]{\bf Question}

\newtheorem*{clai-nn}{Claim}
\newtheorem*{theorem*}{Theorem}
\newtheorem*{remark*}{Remark}
\newtheorem*{ext_probl}{Extension Problem}

\newtheorem*{schoencond*}{Sch\"onflies Condition}

\begin{document}

\title[]{On Continuous Extension of Conformal Homeomorphisms of Infinitely Connected Planar Domains 
}

\author[J. Luo]{Jun Luo} \address{School of Mathematics\\
    Sun Yat-Sen University\\ Guangzhou 512075, China}
\email{luojun3@mail.sysu.edu.cn}

\author[X.T. Yao]{Xiaoting Yao} \address{School of Mathematical Sciences\\
    Fudan University\\ Shanghai 200433, China}
\email{yaoxiaoting@fudan.edu.cn\ (corresponding author)}


\begin{abstract}
We consider conformal homeomorphisms $\varphi$ of generalized Jordan domains $U$ onto planar domains $\Omega$ 
that satisfy both of the next two conditions: (1)  at most countably many boundary components of $\Omega$ are non-degenerate and their diameters have a finite sum;  (2) either the degenerate boundary components of $\Omega$ or those of $U$ form a set of sigma-finite linear measure. We prove that  $\varphi$  continuously extends to the closure of $U$ if and only if every boundary component of $\Omega$ is locally connected. This generalizes the Carath\'eodory's Continuity Theorem and leads us to a new generalization of the well known Osgood-Taylor-Carath\'eodory Theorem. There are three issues that are noteworthy. Firstly, none of the above conditions (1) and (2) can be removed.
Secondly, 
our results remain valid for non-cofat domains and do not follow from the extension results, of a similar nature, that are obtained in very recent studies on the conformal rigidity of  circle domains.
Finally, when $\varphi$ does extend continuously to the closure of $U$, the boundary of $\Omega$ is a Peano compactum.
Therefore, we also show that the following properties are equivalent for any planar domain $\Omega$:

(1) The boundary of $\Omega$ is a Peano compactum.

(2) $\Omega$ has Property S.

(3) Every point on the boundary of $\Omega$ is locally accessible.

(4) Every point on the boundary of $\Omega$ is locally sequentially accessible.

(5) $\Omega$ is finitely connected at the boundary.

(6) The completion of $\Omega$ under the Mazurkiewicz distance is compact.
\end{abstract}


\maketitle

\newpage


\section{Introduction}


We consider conformal homeomorphisms $\varphi: U\rightarrow\Omega\subset\widehat{\mathbb{C}}$ of generalized Jordan domains $U$. Such a domain has at most countably many non-degenerate boundary components $\gamma_k(k\ge1)$, each of which is a Jordan curve, such that the diameters of $\gamma_k$ converge to zero. Our aim is to find conditions under which $\varphi$  has a continuous extension $\overline{\varphi}: \overline{U}\rightarrow\overline{\Omega}$.

To start off, we recall some well known results concerning the special case that $U$ is actually a
circle domain, {\em i.e. a connected open subset of $\hat{\mathbb{C}}$ whose complementary components are points or (geometric) disks}.
When $U$ is an open disk, the Carath\'eodory's {\bf Continuity Theorem} has provided the earliest criterion for the existence of $\overline{\varphi}$ \cite{Caratheodory13-a}.  See also \cite[Theorem 3]{Arsove68-a} or \cite[p.  18]{Pom92}.
\begin{Thm}[{\bf Continuity Theorem}]\label{CCT}
A conformal homeomorphism $\varphi:\mathbb{D}\rightarrow \Omega\subset\widehat{\mathbb{C}}$ of the unit disk $\mathbb{D}=\{z: |z|<1\}$ has a continuous extension $\overline{\varphi}: \overline{\mathbb{D}}\rightarrow\overline{\Omega}$  if and only if $\partial\Omega$ is locally connected.
\end{Thm}

\begin{remark*}
Recall that a {\bf Peano continuum} is the image of  $[0,1]$ under a continuous map.
By the Hahn-Mazurkiewicz-Sierp\'inski Theorem \cite[\S50, II, Theorem 2]{Kuratowski68}, a continuum is locally connected if and only if it is a Peano continuum. Therefore,  the {\bf Continuity Theorem} may be restated as follows: a conformal homeomorphism $\varphi:\mathbb{D}\rightarrow \Omega$ of the unit disk $\mathbb{D}$ onto a planar domain $\Omega$ continuously extends to $\overline{\mathbb{D}}$ if and only if $\partial \Omega$ is a Peano continuum.
\end{remark*}

If $\Omega$ in the above theorem is a {\bf Jordan domain}, so that $\partial\Omega$ is a {\bf Jordan curve}, the extension  $\overline{\varphi}: \overline{\mathbb{D}}\rightarrow\overline{\Omega}$  is actually a homeomorphism. This has been obtained earlier by Osgood and Taylor
\cite[Corollary 1]{Osgood-Taylor1913} and independently by Carath\'eodory \cite{Caratheodory13-b}. We refer to it as the Osgood-Taylor-Carath\'eodory Theorem, or shortly the {\bf OTC Theorem}. See for instance \cite[Theorem 4]{Arsove68-a}.
\begin{Thm}[{\bf OTC Theorem}]\label{OTC-0}
A conformal homeomorphism $\varphi:\mathbb{D}\rightarrow \Omega\subset\hat{\mathbb{C}}$ has a continuous and injective extension to  $\overline{\mathbb{D}}$  if and only if the boundary $\partial\Omega$ is a Jordan curve.
\end{Thm}

There are very recent generalizations of the {\bf OTC Theorem}. In order to present how they differ from what we obtain, especially Theorem \ref{OTC-b}, these results are cited as Theorems \ref{OTC-countable} to \ref{OTC-3} in the sequel.
The first one is a direct corollary of \cite[Theorem 3.2]{He-Schramm93}.

\begin{Thm}[{\bf OTC Theorem} {--- for Countably Connected Circle Domains}]\label{OTC-countable}
Every conformal homeomorphism $\varphi: D\rightarrow\Omega$ of a countably connected circle domain $D$ onto a circle domain $\Omega$ extends to be a homeomorphism between $\overline{D}$ and $\overline{\Omega}$.
\end{Thm}

The second one  includes Theorem \ref{OTC-countable} as a special case.
\begin{Thm}[{\bf\cite[Theorem 2.1]{He-Schramm94}}]\label{OTC-1}
Let $D$ be a circle domain in $\hat{\mathbb{C}}$ whose boundary has $\sigma$-finite linear measure. Let $\Omega$ be another circle domain and let  $\varphi: D\rightarrow\Omega$ be a conformal homeomorphism. Then $\varphi$ extends to be a homeomorphism $\overline{\varphi}: \overline{D}\rightarrow\overline{\Omega}$.
\end{Thm}

The next extension result makes no assumption on the linear measure of $\partial D$, although the circle domain $D$ is additionally required to satisfy the so-called {\bf quasihyperbolic condition}. Here the codomain $\Omega$ is also required to be a {\bf cofat} generalized Jordan domain. On the other hand, Theorem \ref{OTC-b} concerns generalized Jordan domains that may or may not be cofat.

\begin{Thm}[{\bf\cite[Theorem 6.1]{Ntalampekos-Younsi20}}]\label{OTC-2}
Let $D$ be a circle domain with $\infty\in D$ and let $h$ be a conformal map from $D$ onto another domain $\Omega$ with $\infty=h(\infty)\in\Omega$. Suppose that $D$  satisfies the quasihyperbolic condition and that the complementary components of $\Omega$ are points and uniformly fat closed Jordan domains. Then $h$ extends to be a homeomorphism from $\overline{D}$ onto $\overline{\Omega}$.
\end{Thm}

The last one, inferred from \cite[Theorem 6.2]{Schramm95}, generalizes Theorem \ref{OTC-0}.
\begin{Thm}[{\bf OTC Theorem} {--- for Cofat Generalized Jordan Domains}]\label{OTC-3}
Let $\varphi: \Omega_1\rightarrow \Omega_2$ be a conformal homeomorphism between generalized Jordan domains that are countably connected.
If $\Omega_1$ and $\Omega_2$ are both  {\bf cofat} then  $\varphi$ extends to a homeomorphism between $\overline{\Omega_1}$ and $\overline{\Omega_2}$.
\end{Thm}

Theorems \ref{OTC-countable} to \ref{OTC-3} provide fundamental generalizations of
the {\bf Continuity Theorem} and the {\bf OTC Theorem}. These results are closely related to the conformal rigidity for circle domains that satisfy certain conditions \cite{He-Schramm93,He-Schramm94,Ntalampekos-Younsi20}.
In the present paper, we deal with the problem of continuous extension for conformal homeomorphisms $\varphi: U\rightarrow\Omega$ from a generalized Jordan domain onto a planar domain. On the one hand, the underlying domains may have uncountably many boundary components. On the other, unlike in Theorems \ref{OTC-countable} to \ref{OTC-3}, we neither assume  $U$ to be a circle domain nor require the codomain $\Omega$ to be {\bf cofat}.
Thus the generalizations we obtain  do not follow from Theorems \ref{OTC-countable} to \ref{OTC-3}.
Here a domain $\Omega\subset\widehat{\mathbb{C}}$ is said to be cofat provided that there exists a number $\tau>0$ such that every component $P$ of $\hat{\mathbb{C}}\setminus\Omega$ is either a point or is {\bf $\tau$-fat}, in the sense that for any point $x$ in $P\cap\mathbb{C}$ the inequality
$\displaystyle \frac{{\rm Area}\left(P\cap D_r(x)\right)}{{\rm Area}\left(D_r(x)\right)}\ge\tau$ holds for any disk $D_r(x)$ centered at $x$ that does not contain $P$.
Clearly, a convex (closed) domain is fat. See the following Figure \ref{fat} for simple fat and non-fat domains.
\begin{figure}[ht]
\begin{tabular}{ccc}
\begin{tikzpicture}[x=0.8cm,y=1.5cm,scale=0.382]

\fill [gray, thick, domain=-4:0,samples=2000] plot (\x, {sqrt(4-(\x+2)*(\x+2))});
\fill [gray, thick, domain=0:4,samples=2000] plot (\x, {sqrt(4-(\x-2)*(\x-2))});
\fill [gray, thick, domain=-4:4,samples=2000] plot (\x, {0.1-sqrt(16-\x*\x)});
\end{tikzpicture}
&

\begin{tikzpicture}[x=1cm,y=1cm,scale=0.764]

\fill[gray] (0,0) circle(2);

\end{tikzpicture}

& \begin{tikzpicture}[x=1cm,y=1cm,scale=0.382]

\fill[gray] (0,0) circle(2.5);
\draw[gray,thin] (-0.6,2.0) arc (345:360:17cm);
\draw[gray,thin] (-0.013,6.2) arc (180:195:17cm);

\draw[gray,thick] (-0.58,1.8) arc (345:360:16.5cm);
\draw[gray,thick] (-0.01,5.7) arc (180:195:17cm);

\fill[gray] (-0.6,1.5) -- (-0.02,5.35) -- (0.55,1) -- (-0.6,1);
\fill[gray] (-0.68,1.8) -- (-0.4,2.8) -- (-0.03,3.8) -- (0.1,3.8) -- (0.6,1.8) -- (-0.6,1.8);
\end{tikzpicture}
\end{tabular}
\caption{A non-convex fat domain, a disk, and a non-fat domain.}\label{fat}
\end{figure}

We will discuss the following.

\begin{ext_probl}
Under what conditions does a conformal homeomorphism  $\varphi:U\rightarrow \Omega\subset\widehat{\mathbb{C}}$ of a circle domain (or generalized Jordan domain) have a continuous extension $\overline{\varphi}: \overline{U}\rightarrow\overline{\Omega}$ ?
\end{ext_probl}

There are basic observations and fundamental examples that shall orient our discussions along proper directions. Firstly,
even if  $\varphi: U\rightarrow\Omega$ is just a homeomorphism, it is routine to check that there is a one-to-one correspondence between the boundary components of $U$ and those of $\Omega$. Following \cite{He-Schramm93}, we denote this correspondence by  $\varphi^B$. Recall that
the non-degenerate boundary components of a generalized Jordan domain $U\subset\hat{\mathbb{C}}$ are Jordan curves whose diameters form a {\bf null sequence}, {\em i.e.} a finite sequence or an infinite one converging to zero. If $\varphi$ does have  a continuous extension $\overline{\varphi}: \overline{U}\rightarrow\overline{\Omega}$,  every component of $\partial\Omega$ is a Peano continuum, possibly a point; moreover, the uniform continuity of $\overline{\varphi}$ requires that the {\bf diameters} of the non-degenerate components of $\partial\Omega$ form a null sequence, too. This indicates that $\partial\Omega$ is a Peano compactum.
Here a  Peano compactum means a compact metrisable space whose components are each a Peano continuum, possibly a single point, such that for any $C>0$ at most finitely many of the components are of diameter $>C$. See \cite[Theorems 1 to 3]{LLY-2019} concerning how this notion  arises from the models discussed in \cite{BCO11,BCO13}. These models have two aims. One is to describe the Julia set $J$ of a rational map $f:\widehat{\mathbb{C}}\rightarrow \widehat{\mathbb{C}}$ from a topological point of view; the other is to analyse the dynamics of the restriction $\left.f\right|_J$.

Accordingly, in the above extension problem we shall add the following.
\\
{\bf(Assumption 1)}. The boundary  $\partial\Omega$ is a Peano compactum.

Secondly, with {\bf Assumption 1} the diameters of the non-degenerate components of $\partial\Omega$ form a null sequence, denoted $\{\delta_n: n\ge1\}$. However, if $\delta_n$ does not converge to zero sufficiently fast (for instance, such that $\sum_n\delta_n=\infty$), then it is possible that $\varphi^B$ maps a point component of $\partial U$ to a non-degenerate component of $\partial\Omega$. See {\bf Example \ref{Bishop_Example}} for the concrete construction of a domain $\Omega\subset\{z: 1<|z|<2\}$ with countably many boundary components, one of which is $\partial\mathbb{D}=\{z: |z|=1\}$.
Due to the advancements made by He and Schramm in \cite{He-Schramm93}, the domain $\Omega$ allows the existence of a conformal homeomorphism $\varphi: D\rightarrow\Omega$ defined on a circle domain $D$.
We will show that for any conformal homeomorphism $\varphi: D\rightarrow\Omega$ of a circle domain $D$ onto $\Omega$ the boundary correspondence $\varphi^B$ maps a point component of $\partial D$ to $\partial\mathbb{D}$. Therefore,  in the above extension problem we also add the following.
\\
{\bf(Assumption 2)}. The diameters of the non-degenerate components of $\partial\Omega$ have a finite sum.

Thirdly, if the point components of $\partial U$ form a set that is too large, say having positive $2$-dimensional Lebesgue measure, then it is possible that $\varphi^B$ maps a point component of $\partial U$ to a non-degenerate component of $\partial\Omega$. See for instance \cite[Theorem 4.1]{Gehring-Martio} for a conformal homeomorphism $\varphi:D\rightarrow D^*$ between circle domains such that the boundary components of $D$ are each a singleton and exactly one component of $\partial D^*$ is a circle.
It is known that there is no continuous extension  $\overline{\varphi}: \overline{D}\rightarrow\overline{D^*}$.
To avoid such situations, and with motivations from discussions in
\cite{He-Schramm94} on the conformal rigidity of circle domains, we further add the following.
\\
{\bf(Assumption 3)}. The point components of $\partial U$ form a set of sigma-finite linear measure.

Now we are well prepared to present our first generalization of the {\bf Continuity Theorem}, to conformal homeomorphisms of a circle domain that may be uncountably connected.
\begin{main-theorem}
[{\bf First Generalization of Continuity Theorem}]
\label{arsove-1}
Let  $\Omega\subset\widehat{\mathbb{C}}$ be a domain satisfying  {\bf Assumptions 1-2}.
Let $D\subset\widehat{\mathbb{C}}$ be a circle domain satisfying {\bf Assumption 3}. Then any conformal homeomorphism $\varphi: D\rightarrow \Omega$ has a continuous extension  $\overline{\varphi}: \overline{D}\rightarrow\overline{\Omega}$.
\end{main-theorem}

A key step in the proof for Theorem \ref{arsove-1} is to find  upper bounds for the oscillations of $\varphi$, in which {\bf Assumption 3} plays a key role. Particularly,
from  \cite[Lemma 1.3]{He-Schramm94} (Lemma \ref{HS-1994_1.3} in this paper) we can infer Lemma \ref{small_level_sets},
which is  then used to construct upper bounds for the oscillations of $\varphi$ as required. See Theorem \ref{oscillation2}.
With the help of  Lemma \ref{Wolff_lemma_2},  we obtain another generalization of the {\bf Continuity Theorem} by replacing {\bf Assumption 3} with the following.
\\
{\bf(Assumption 3*)}. The point components of $\partial\Omega$ form a set of sigma-finite linear measure.

\begin{main-theorem}
[{\bf Second Generalization of Continuity Theorem}]
\label{arsove-2}
Let  $\Omega\subset\widehat{\mathbb{C}}$ be a domain satisfying  {\bf Assumptions 1-2} and {\bf3*}. Then any conformal homeomorphism $\varphi: D\rightarrow \Omega$ of a circle domain $D$ onto $\Omega$ has a continuous extension  $\overline{\varphi}: \overline{D}\rightarrow\overline{\Omega}$,
\end{main-theorem}

Two issues concerning {\bf Theorems \ref{arsove-1}-\ref{arsove-2}} are noteworthy.
Firstly,  {\bf  Theorem \ref{arsove-1}} {\rm(or {\bf\ref{arsove-2}})} actually indicates that, under {\bf Assumptions 2-3} {\rm(or  {\bf 2-3*})}, any conformal homeomorphism $\varphi: D\rightarrow\Omega$ of a circle domain $D$ onto $\Omega$ has a continuous extension $\overline{\varphi}: \overline{D}\rightarrow\overline{\Omega}$ if and only if $\partial\Omega$ is a Peano compactum.
Secondly, we may replace the circle domain $D$ by a {\bf generalized Jordan domain}.
See  Theorems \ref{arsove_sigma-new} and \ref{arsove-new}.
Therefore, we can obtain the following.

\begin{main-theorem}[{\bf Generalized OTC Theorem}]\label{OTC-b}
Let $\Omega_1,\Omega_2$ be generalized Jordan domains. Let $\{P_n\}$ (respectively, $\{Q_n\}$) denote the  non-degenerate  components of $\partial\Omega_1$ (respectively, of $\partial\Omega_2$). Then any conformal homeomorphism $f:\Omega_1\rightarrow\Omega_2$ extends to a continuous map from $\overline{\Omega_1}$ onto $\overline{\Omega_2}$ provided that the following two conditions are both satisfied:
\begin{itemize}
\item[(i)] The inequality $\sum_n{\rm diam}(Q_n)<\infty$ holds for the diameters of $Q_n$.
\item[(ii)] The point components of $\partial \Omega_1$ or those of $\partial\Omega_2$ form a set of $\sigma$-finite linear measure.
\end{itemize}
If, in addition, we have $\sum_n{\rm diam}(P_n)<\infty$ then $f$ extends to a homeomorphism from $\overline{\Omega_1}$ onto $\overline{\Omega_2}$. Consequently, this extension further extends to a homeomorphism of $\widehat{\mathbb{C}}$ onto itself. \end{main-theorem}

\begin{remark*} Theorem \ref{OTC-b} differs from Theorem \ref{OTC-3}.
Firstly,  Theorem \ref{OTC-3} concerns countably connected domains while Theorem \ref{OTC-b} also applies to uncountably connected domains. Secondly, the {\bf cofatness} of $\Omega_1$ and $\Omega_2$ in Theorem \ref{OTC-3} is replaced by two properties: (1) for one of them  the point components of the boundary form a set of {\bf sigma-finite linear measure} and (2) for both of them the {\bf diameters} of the non-degenerate boundary components  have a finite sum.
Theorefore, Theorem \ref{OTC-b} neither implies nor follows from any of Theorems \ref{OTC-countable} to \ref{OTC-3}.
\end{remark*}

In our proofs for {\bf Theorems \ref{arsove-1}} and {\bf\ref{arsove-2}}, we employ a fundamental fact that connects the {\bf Property S} of a planar domain to the topology of its boundary. Here, a metric space $X$ is said to have Property S provided that for each $\epsilon>0$, the set $X$ is the union of finitely many connected sets of diameter less than $\epsilon$ \cite[p.  20]{Whyburn42}.

In Theorem \ref{property_s} we characterize all planar domains with Property S  as those whose boundaries are Peano compacta. This generalizes \cite[p.  112, Theorem (4.2)]{Whyburn42}, from simply connected domains to all planar domains.
There are more criterions, which are summarized in the following.

\begin{main-theorem}\label{topology_metric}
Let $\Omega\subset\hat{\mathbb{C}}$ be a domain. The following are equivalent:

(1) $\partial \Omega$ is a Peano compactum;

(2) $\Omega$ has Property S;

(3) Every point of $\partial \Omega$ is locally accessible;

(4) Every point of $\partial \Omega$ is locally sequentially accessible;

(5) $\Omega$ is finitely connected at the boundary;

(6) The completion of $\Omega$ under the Mazurkiewicz distance $d$ is compact.
\end{main-theorem}

\begin{remark*}
Theorem \ref{topology_metric} demonstrates an interplay between the topology of $\Omega$, that of  $\partial\Omega$, and the completion of the metric space $(\Omega,d)$. For the definition of Mazurkiewicz distance $d$ one may refer to Remark \ref{notions_for_th.1}. For a special case of  Theorem \ref{topology_metric}, when $\Omega$ is simply connected, one may refer to \cite{Herron12}.
Note that Theorem \ref{topology_metric} is also motivated by and provides new generalizations of known criteria for a simply connected planar domain to have Property $S$. See \cite[p.  112, Theorem (4.2)]{Whyburn42}. This theorem will be cited fully in the current paper  as Theorem \ref{whyburn_112}.
Also note that the completion of $(\Omega,d)$ is compact if and only if $\Omega$ is finitely connected at the boundary. See \cite[Theorem 1.1]{BBS16}. The authors of \cite{BBS16} also obtain the equivalences of (2), (4) and (5) for countably connected planar domains $\Omega$  \cite[Theorem 1.2]{BBS16} or for slightly more general choices of $\Omega\subset\hat{\mathbb{C}}$ \cite[Theorem 4.4]{BBS16}. Theorem \ref{topology_metric} improves those earlier results by obtaining the equivalence of all those properties for an arbitrary planar domain $\Omega$.
\end{remark*}

The other parts of the paper are arranged as follows.
In Section \ref{cluster_set} we have two aims.
One is to obtain Theorem \ref{characterizing_GJD}, which  characterizes all generalized Jordan domains $U$, in terms of (1) the simple connectedness of $U$ at the boundary and (2) the connectedness of all cluster sets for any continuous map from $U$ to a compact metric space.
The other is to obtain  Theorem \ref{topology_metric}.

In
Section \ref{proof_1/2} we give topological counterparts for {\bf Theorems \ref{arsove-1}-\ref{OTC-b}}. See Theorem \ref{topological-cct}. Here we also analyze the monotonicity and injectivity of the continuous extension $\overline{\varphi}: \overline{U}\rightarrow\overline{\Omega}$ of a homeomorphism $\varphi$ from a generalized Jordan domain onto a domain $\Omega\subset\widehat{\mathbb{C}}$, if it exists. In Theorem \ref{monotone_bd_map}, we show that  $\overline{\varphi}$ is monotone if and only if $\Omega$ is also a generalized Jordan domain.

In Section \ref{outline-cct} we will prove Theorems  \ref{arsove-1} to \ref{OTC-b}.  These proofs  are both based on certain estimates from above of the oscillations of some conformal homeomorphism $\varphi: U\rightarrow\Omega$, from a generalized Jordan domain $U$ onto a domain $\Omega\subset\widehat{\mathbb{C}}$, so that Theorem \ref{topological-cct} may be applied.

In Section \ref{Example} we concretely construct a countably connected domain $\Omega\subset\mathbb{C}$ such that there are conformal homeomorphisms $\varphi$ from a circle domain $D$ onto $\Omega$ and the boundary map $\varphi^B$ sends a point component of $\partial D$ to a non-degenerate component of $\partial\Omega$. See Example \ref{Bishop_Example}. Moreover, we may construct this domain $\Omega$ as a generalized Jordan domain. See Remark \ref{Bishop_Example-b}.

\section{Theory of Cluster Sets for Generalized Jordan Domains}\label{cluster_set}
This section has two aims. One is to characterize generalized Jordan domains in terms of the connectedness of cluster sets. The other is to find necessary and sufficient conditions for a planar domains to have property $S$, that is, to prove Theorem \ref{topology_metric}.
Here we recall that the boundary components of a generalized Jordan domain are points or Jordan curves and that there are at most countably many Jordan curves, whose diameters converge to zero.
Moreover, for any continuous map $f: U\rightarrow X$ of a planar domain $U$ to a compact metric space $X$,  the {\bf cluster set} at $z_0\in\partial U$, denoted as  $C(f,z_0)$, consists of all $w$ such that there exists an infinite sequence $z_k\in U$ satisfying $z_k\rightarrow z_0$ and $f(z_k)\rightarrow w$. In other words, we have
\[C(f,z_0):=\bigcap_{r>0}\overline{f(D_r(z_0)\cap U)}.\]
Clearly, every cluster set is nonempty and closed, thus is also a compactum. For the basics of cluster sets, one may refer to \cite{CL66}.

To this end, we shall obtain  as well as the following.
\begin{theo}\label{characterizing_GJD}
The following statements are equivalent for any domain $U\subset\widehat{\mathbb{C}}$:
(1) $U$ is a generalized Jordan domain, (2) $U$ is simply connected at the boundary, (3) for any continuous map $f$ of $U$ into a compact metric space $X$ all cluster sets $C(f,z)$ with $z\in\partial U$ are connected.
\end{theo}

The property of being {\bf simply connected at the boundary}, to be defined as follows, is a special sub-case for the property of being finitely connected at the boundary, which is to be discussed slightly later in Remark \ref{notions_for_th.1}, before we start to prove Theorem \ref{topology_metric}.
\begin{deff}\label{simply_connected}
A domain $U\subset\widehat{\mathbb{C}}$ is said to be {\bf simply connected at a point $z_0\in\partial U$} provided that for any disk $D_r(z_0)$ centered at $z_0$  with radius $r>0$ there is a set $N_r$ open in $\hat{\mathbb{C}}$ with $z_0\in N_r\subset D_r(z_0)$ such that $N_r\cap U$ is connected.
\end{deff}

Note that for any continuous map $f: U\rightarrow X$ of a  domain $U\subset\widehat{\mathbb{C}}$ into a compact metric space, if $U$ is simply connected at $z_0\in\partial U$ then the cluster set $C(f,z_0)$ is equal to  $\displaystyle\bigcap_{r>0}\overline{f(N_r\cap U)}$, where $N_r$ is chosen  as in Definition  \ref{simply_connected}. So the implication $(2)\Rightarrow(3)$ in Theorem \ref{characterizing_GJD} is immediate. Therefore, we only need to consider the implications $(1)\Rightarrow(2)$ and $(3)\Rightarrow(1)$, which are respectively done in Theorems \ref{connected_cluster} and \ref{connected_cluster_GJD}.

\begin{theo}\label{connected_cluster}
Each generalized Jordan domain is simply connected at the boundary.
\end{theo}

\begin{theo}\label{connected_cluster_GJD}
If a domain $U\subset\widehat{\mathbb{C}}$ is not a generalized Jordan domain, there is a continuous map $f$ of $U$ to a compact metric space $X$ such that $C(f,z_0)$  is disconnected for some   $z_0\in\partial U$.
\end{theo}

We divide the rest of this section into three parts. Firstly, we obtain an intrinsic connection between a metric property of a planar domain $\Omega$, the so-called ``Property S'', and the topology of its boundary $\partial\Omega$. See Theorem \ref{property_s}.  Secondly, with the help of this connection, we continue to prove Theorems \ref{connected_cluster} to \ref{connected_cluster_GJD}. Finally, we prove Theorem  \ref{topology_metric}.

\subsection{Property $S$ and Peano Compactum}\label{s-pc}

For planar domains, Property $S$ and the property of having a Peano compactum as its boundary are closely connected. A special case of this connection is already known. Actually, by \cite[p.  112, Theorem (4.2)]{Whyburn42}, we have.
\begin{theo}\label{whyburn_112}
If $\Omega\subset\hat{\mathbb{C}}$ is a domain whose boundary is a
continuum
 the following are equivalent:
\begin{itemize}
\item[(i)] $\Omega$ has Property $S$,
\item[(ii)] every point of $\partial\Omega$ is locally accessible from $\Omega$,
\item[(iii)] every point of $\partial\Omega$ is accessible from all sides from $\Omega$,
\item[(iv)] $\partial\Omega$ is locally connected, or equivalently, a Peano continuum.
\end{itemize}
\end{theo}

Here a metric space $X$ is said to have {\bf Property S} provided that for each $\epsilon>0$ the set $X$ is the union of finitely many connected sets of diameter less than $\epsilon$ \cite[p.  20]{Whyburn42}.
Moreover, a point $p\in\partial\Omega$ is said to be {\bf locally accessible from $\Omega$} provided that for any $\epsilon>0$ there is a number $\delta>0$ such that for any $x\in\Omega$ with $|x-p|<\delta$ one can find a simple arc $\overline{xp}\subset \Omega\cup\{p\}$ that joins $x$ to $p$ and has a diameter smaller than $\epsilon$ \cite{Arsove67}. Also note that if a point $x\in\partial\Omega$ is locally accessible it is also said to be {\bf regularly accessible} \cite[p.  111]{Whyburn42} .

\begin{theo}\label{property_s}
A domain $\Omega\subset\hat{\mathbb{C}}$ has Property $S$ if and only if $\partial\Omega$ is a Peano compactum.
\end{theo}

Theorem \ref{property_s} is motivated by and also provides a partial generalization for Theorem \ref{whyburn_112}. Here we remove the requirement that $\partial\Omega$ be a continuum and keep  items (ii) and (iii) untouched.
In order to prove Theorem \ref{property_s}, we will use a  notion introduced in \cite{LLY-2019}, the {\bf Sch\"onflies condition} for planar compacta.
\begin{deff}\label{Schonflies_condition}
A compactum $K\subset\mathbb{C}$ is said to satisfy the {\bf Sch\"onflies condition} provided that for the closed strip $W=W(L_1,L_2)$ bounded by two arbitrary parallel lines $L_1$ and $L_2$, the {\bf difference} $W\setminus K$ has at most finitely many components intersecting $L_1$ and $L_2$ both.
\end{deff}

The above notion, and related results, has motivations from recently developed topological models that are very helpful in the study of polynomial Julia sets \cite{BCO11,BCO13,Curry10,Kiwi04}. These models also date back to the 1980's, when Thurston and Douady and their colleagues started applying the {\bf Carath\'eodory's Continuity Theorem} to the study of polynomial Julia sets, which are assumed to be connected and locally connected. See for instance \cite{Douady93} and \cite{Thurston09}.

Notice that  a compact set $K\subset\mathbb{C}$ satisfies the Sch\"onflies condition if and only if it is a Peano compactum \cite[Theorem 3]{LLY-2019}.
This happens if and only if the {\bf intersection} $W\cap K$ has at most finitely many components intersecting $L_1$ and $L_2$ both, with $W=W(L_1,L_2)$ defined in the same way. See \cite[Lemma 3.8(3)]{LLY-2019}. Also note that the Sch\"onflies condition involves parallel lines, so we stay on the complex plane $\mathbb{C}$, not on the extended complex plane $\hat{\mathbb{C}}$. On the other hand,  we can obtain a similar criterion by replacing $W$ with a closed annulus determined by two arbitrary disjoint Jordan curves $J_1,J_2\subset\hat{\mathbb{C}}$ which is the closure of $U=U(J_1,J_2)$,  the component of $\hat{\mathbb{C}}\setminus(J_1\cup J_2)$ with $\partial U=J_1 \cup J_2$. Such a criterion then works for all compacta lying on $\hat{\mathbb{C}}$.  See for instance \cite[Remark 3.9]{LLY-2019}. In other words, we shall have the following lemma, which is also a direct corollary of \cite[Theorem 7]{LLY-2019}.

\begin{lemm}
\label{R_K}
A compact set $K\subset\hat{\mathbb{C}}$ is a Peano compactum if and only if for any two disjoint Jordan curves $J_1,J_2$ either of the next two conditions is satisfied: (1)  $\overline{U(J_1,J_2)}\setminus K$ has at most finitely many components that intersect each of $J_1$ and $J_2$ both; (2)  $\overline{U(J_1,J_2)}\bigcap K$ has at most finitely many components that intersect each of $J_1$ and $J_2$ both.
\end{lemm}

\begin{proof}[{\bf Proof for Theorem \ref{property_s}}] Let us start from a proof for the ``only if'' part. With no loss of generality, let us assume $\infty\in\Omega$ and fix a closed disk $D_0$ centered at $\infty$ and lying in $\Omega$. Then $\Omega$ as a domain in $\hat{\mathbb{C}}$ has Property $S$ if and only if $\Omega_1=\Omega\setminus D_0$ as a domain in $\mathbb{C}$ has; moreover, $\partial\Omega$ is a Peano compactum if and only if $\partial\Omega_1$ is.

If $\partial\Omega$ is not a Peano compactum, then neither is $\partial\Omega_1$. So there are two parallel lines $L_1,L_2$ such that for the {\bf closed strip} $W=W(L_1,L_2)$ bounded by $L_1$ and $L_2$, {\em i.e.} the closure of the component of $\mathbb{C}\setminus(L_1\cup L_2)$ whose boundary  is $L_1\cup L_2$, the {\bf difference} $W\setminus \partial\Omega_1$ has infinitely many components intersecting both $L_1$ and $L_2$. Denote these components as $W_1,W_2,\ldots$.

All those $W_n$ are locally compact and locally connected. So each of them is arcwise connected \cite[p.  38, (5.2)]{Whyburn42}. Thus we may choose open arcs $\alpha_n\subset W_n$ joining a point $a_n$ on $W_n\cap L_1$ to a point $b_n$ on $W_n\cap L_2$. Here an open arc means a simple arc without the two end points. Renaming the arcs $\alpha_n$ and going to an appropriate subsequence, if necessary, we may assume that for any $n>1$ the two arcs $\alpha_{n-1}$ and $\alpha_{n+1}$ lie in different components of $W\setminus\alpha_n$. It is clear that the arcs $\alpha_n$ may be arranged inside $W$ {\bf linearly} from left to right.  See Figure \ref{long-band} for a simplified depiction of this arrangement.
\begin{figure}[ht]
\begin{center}
\vskip -0.5cm
\begin{tikzpicture}[scale=0.618,x=1cm,y=0.75cm]

\foreach \i in {1,...,50}
{
\draw[line width=2 pt ,color=gray!20](3+3*\i/50,5) -- (2+3*\i/50,3) -- (3.5+3*\i/50,1.5) -- (3+3*\i/50,0);
\draw[line width=2 pt ,color=gray!20](11+3*\i/50,5) -- (10+3*\i/50,3) -- (11.5+3*\i/50,1.5) -- (11+3*\i/50,0);
}

\draw  (-2,0)--(22,0);
\draw  (-2,5)--(22,5);
\draw[line width=2 pt ,color=gray](3,5) -- (2,3) -- (3.5,1.5) -- (3,0);
\draw[line width=2 pt ,color=gray] (6,5) -- (5,3) -- (6.5,1.5) -- (6,0);
\draw[line width=2 pt ,color=gray] (11,5) -- (10,3) -- (11.5,1.5) -- (11,0);
\draw[line width=2 pt ,color=gray] (14,5) -- (13,3) -- (14.5,1.5) -- (14,0);

\node at (22.75,5) {$L_1$}; \node at (22.75,0) {$L_2$};
\node at (1.4,2.9) {$\alpha_1$};  \node at (3.75,2.9) {$D_1$}; \node at (11.75,2.9) {$D_n$};
\node at (6.0,2.9) {$\alpha_2$};
\node at (9.4,2.9) {$\alpha_n$};
\node at (14.5,2.9) {$\alpha_{n+1}$};
\node at (11,5.5) {$a_n$};  \node at (11,-0.5) {$b_n$};
\node at (14.5,5.5) {$a_{n+1}$};  \node at (14.5,-0.5) {$b_{n+1}$};
\node at (9.0,1.4) {\Large$\cdots\cdots$};
\node at (16.5,1.4) {\Large$\cdots\cdots$};

\draw [fill=black] (11,0) circle (0.1);
\draw [fill=black] (11,5) circle (0.1);
\draw [fill=black] (14,0) circle (0.1);
\draw [fill=black] (14,5) circle (0.1);
\draw [fill=black] (6,0) circle (0.1);
\draw [fill=black] (6,5) circle (0.1);
\draw [fill=black] (3,0) circle (0.1);
\draw [fill=black] (3,5) circle (0.1);
\end{tikzpicture}
\end{center}
\vskip -0.618cm
\caption{The arcs $\alpha_n$ and the Jordan domains $D_n$.}\label{long-band}
\vskip -0.5cm
\end{figure}

Now, let $D_n (n\ge1)$ be the bounded component of $\mathbb{C}\setminus(L_1\cup L_2\cup \alpha_n\cup\alpha_{n+1})$. Then each $D_n$ is a Jordan domain; moreover, the closure $\overline{D_n}$ contains a continuum $M_n\subset\partial\Omega_1$ that separates $\alpha_n$ from $\alpha_{n+1}$ in $\overline{D_n}$ \cite[lemma 3.6]{LLY-2019}. Such a continuum $M_n$ must intersect both $L_1$ and $L_2$. So we can
choose $x_n\in M_{2n-1}$ for all $n\ge1$ with
${\rm dist}(x_n,L_1)={\rm dist}(x_n,L_2)=\frac12\text{\rm dist}(L_1,L_2)$. Let $\epsilon>0$ be a number smaller than $\frac14\text{\rm dist}(L_1,L_2)$. Since $x_n\in \left(D_{2n-1}\cap\partial\Omega_1\right)$ we may find a point $y_n\in \left(D_{2n-1}\cap\Omega_1\right)$ satisfying $|x_n-y_n|<\epsilon$. Clearly, for any $m,n\ge1$ the two points $y_n, y_{n+m}\in \Omega_1$ are separated in $\overline{W}$ by $M_{2n}$, which lies in $\overline{D_{2n}}\cap\partial\Omega_1$. In other words,  we have obtained an infinite set $\{y_n\}$ of points in $\Omega_1$, no two of which may be contained in a single connected subset of $\Omega_1$ that has diameter $<\epsilon$. In other words,  $\Omega_1$ and hence $\Omega$ does not have Property $S$.

In the sequel, we prove the ``if'' part. Suppose that $\Omega$ does not have Property $S$. Then we can find $\displaystyle\epsilon\in\left(0,\frac12{\rm diam}(\Omega)\right)$, such that $\Omega$ is not the union of finitely many connected subsets of diameter $\le4\epsilon$. Let $x_1,x_2\in\Omega$ be two points with $|x_1-x_2|>2\epsilon$. Then $A_2=\{x_1,x_2\}$ has  the next property (*): {\em no pair of its points is contained in a single connected subset of $\Omega$ with diameter  $\le2\epsilon$}. Once $A_n=\{x_1,\ldots,x_n\}$ has been chosen so that it has the property (*), we claim that there always exists a further point $x_{n+1}$ in $\Omega\setminus A_n$ such that $A_{n+1}=A_n\cup\{x_{n+1}\}$ has the property (*). If not, for every $x\in\Omega$ there would be a connected set $\Omega_x$, of diameter  $\le2\epsilon$, that contains $x$ and some $x_i\in A_n$.  Let $E_i$ consist of all those $\Omega_x$ that contains $x_i$. Then $E_1,\ldots,E_n$ are connected sets of diameter  $\le4\epsilon$ and their union covers $\Omega$. This is a contradiction.

The above procedure, to choose $x_n$ for all $n\ge3$, may be repeated indefinitely. In this way, we can  find infinitely many points $x_1,x_2,\ldots\in\Omega$ such that no pair of them is contained in a single connected subset of $\Omega$ having diameter  $\le2\epsilon$. By the compactness of  $\overline{\Omega}$ and by going to an appropriate subsequence, if necessary, we may assume that $\lim\limits_{i\rightarrow\infty} x_i=x$. The way we choose the points $x_i$ then implies that $x\in\partial\Omega$.

Given $r\in(0,\epsilon)$, choose $i_0\ge1$ such that $x_i\in D_r(x)$ for all $i\ge i_0$. Fix a point $x_0\in \Omega$ with $|x-x_0|>\epsilon$ and choose arcs $\alpha_i\subset \Omega$ starting from $x_0$ and ending at $x_i$. For any $i\ge i_0$ let $a_i\in\alpha_i$ be the last point at which $\alpha_i$ leaves $\partial D_\epsilon(x)$; let $b_i\in\alpha_i$ be the first point after $a_i$ at which $\alpha_i$ encounters $\partial D_r(x)$.
\begin{figure}[ht]
\begin{center}
\vskip -0.382cm
\begin{tikzpicture}[x=1cm,y=1cm,scale=0.5]
\draw[thick]  (0,0) circle (5);
\draw[thick]  (0,0) circle (2.5);
\node at (6.35,0) {$\partial D_\epsilon(x)$};
\node at (0.75,-3.00) {$\partial D_r(x)$};
\node at (0.75,0) {$x$};
\draw  [fill,black](0,0)circle (0.2);
\node at (-6.75,0) {$x_0$};
\draw  [fill,gray](-6,0)circle (0.2);

\node at (0.8,-1.2) {$x_i$};
\draw  [fill,gray](0,-1.2)circle (0.2);

\draw  [fill,gray](-1.5,-2.0)circle (0.2);
\draw  [fill,gray](-4,-3.0)circle (0.2);
\node at (-4.5,-3.5) {$a_i$};
\node at (-1.35,-1.1) {$b_i$};
\draw[line width=2 pt ,color=gray](-4,-3) -- (-2.75,-3) -- (-1.5,-2);
\node at (-2.25,-3.5) {$\beta_i$};
\end{tikzpicture}
\end{center}
\vskip -0.5cm
\caption{The relative locations for $x_i, a_i, b_i$ and the arcs $\beta_i$. 
}\label{annulus_puzzle}
\end{figure}
See Figure \ref{annulus_puzzle}.  Now let $\beta_i\subset\alpha_i$ be the sub-arc from $a_i$ to $b_i$ and $\gamma_i\subset\alpha_i$ the one from $b_i$ to $x_i$. Since no pair of the points $x_i$ is contained in a single connected subset of $\Omega$ that is of diameter $\le2\epsilon$, we see that all those arcs $\{\beta_i: \ i\ge i_0\}$ are disjoint. Moreover, we can further infer that no two of them may be contained in the same component of $A\setminus\partial\Omega$, where $A$ is the {\bf closed annulus} with boundary circles $\partial D_r(x)$ and $\partial D_\epsilon(x)$. Indeed, if this happens for $\beta_k, \beta_j$ with $k\ne j\ge i_0$ then $\beta_k\cup\beta_j$ lies in a component $P$ of $A\setminus \partial\Omega$, which is a subset of $\Omega$. In such a case  $\gamma_k\cup\beta_k\cup P\cup\beta_j\cup\gamma_j$ is a connected subset of $\Omega$ containing $\{x_k, x_j\}$ that lies in $D_\epsilon(x)$ and hence has a diameter $\le 2\epsilon$. This is prohibited, by the choices of $\{x_i\}$.

Now, let $P_i (i\ge i_0)$ be the component of $A\setminus \partial\Omega$ that contains $\beta_i$. Then $P_i\cap P_j=\emptyset$ for all $i\ne j\ge i_0$ and $A\setminus\partial\Omega$ has infinitely many components that intersect $\partial D_r(x)$ and $\partial D_\epsilon(x)$ both. By Lemma \ref{R_K}, this indicates that $\partial\Omega$ is not a Peano compactum.
\end{proof}

\subsection{Proofs for Theorems \ref{connected_cluster} and \ref{connected_cluster_GJD}}
We start from Theorem \ref{connected_cluster}.

\begin{proof} [{\bf Proof for Theorem \ref{connected_cluster}}] Since a Jordan curve separates $\hat{\mathbb{C}}$ into two domains,  $\partial U$ has at most countably many components that are Jordan curves. Denote these Jordan curves as $\{\Gamma_n\}$.

We shall need {\bf Zoretti's Theorem} \cite[p.  35, Corollary 3.11]{Whyburn64}.

\begin{theorem*}[\bf Zoretti's Theorem]\label{Zoretti}
If $K$ is a component of a compact set $M$ (in the plane) and $\epsilon$ is any positive number, then there exists a simple closed curve $J$ which encloses $K$ and is such that $J\cap M=\emptyset$, and every point of $J$ is at a distance less than $\epsilon$ from some point of $K$.
\end{theorem*}

If $\{z_0\}$ is a point component of $\partial U$, by {\bf Zoretti's Theorem}, we can choose an infinite sequence of Jordan curves $J_i\subset U$ with $i\ge1$ such that the component of $\widehat{\mathbb{C}}\setminus J_i$ containing $z_0$, denoted as $W_i$, also contains $J_{i+1}$. By \cite[p. 117, Theorem 11.2]{Newman} we see that $W_i\cap U$ is connected. This indicates that $U$ is simply connected at $z_0\in\partial U$. In the sequel, we only consider the case that $z_0\in\Gamma_p$ for some $p\ge1$. By the Sch\"onflies Theorem \cite[p.  72, Theorem 4]{Moise}, we may assume that $\Gamma_p=\{|z|=1\}$  and  $U\subset \mathbb{D}^*:=\{|z|>1\}\subset\hat{\mathbb{C}}$.

Given an open subset $V_0$ of $\hat{\mathbb{C}}$ that contains $z_0$, we fix a closed disk $D\subset V_0$ centered at $z_0$.  As $U$ is a generalized Jordan domain, the boundary $\partial U$ is a Peano compactum. By Theorem \ref{property_s}, we see that $U$ has property $S$. Therefore, we may find finitely many domains $M_n(1\le n\le N)$,  of diameter less than ${\rm dist}\left(D,\hat{\mathbb{C}}\setminus V_0\right)$, such that $\bigcup_nM_n=U$ and every $M_n$  has Property $S$. This is possible by \cite[p.  21, Corollary (15.41)]{Whyburn42}, which reads as follows.

\begin{theorem*}
Any metric space having Property $S$ is the sum of a finite number of arbitrarily small connected subsets each having Property $S$. Furthermore these subsets may be chosen either as open sets or as closed sets.
\end{theorem*}

Let $W$ be the union of all those $M_n$ with $z_0\in\overline{M_n}$. Renaming the domains $M_n$, we may find an integer $N_0$ such that $z_0\in\overline{M_n}$ if and only if $1\le n\le N_0$. Therefore, we actually have $W=\bigcup_{n=1}^{N_0}M_n$.

Using {\bf Zoretti's Theorem} repeatedly, we may find a sequence of pairwise disjoint Jordan curves $\gamma_k\subset U$ that converge to $\Gamma_p$ under Hausdorff distance.
Fix a closed subarc $w_0\subset U$ that lies in $\mathbb{D}^*\cap\partial D$. Then $w_0$ separates $\mathbb{D}^*\cap\partial D$ into two open arcs, denoted as $a$ and $b$, each of which connects $w_0$ to $\Gamma_p$. Choosing an appropriate subsequence, we may assume that every $\gamma_k$ separates $w_0$ from $\Gamma_p$ thus intersects both $a$ and $b$. Recall that $D$ is chosen to be a closed disk.  Applying the Cut Wire Theorem \cite[p.  72, Theorem 5.2]{Nadler92}, we see that the compact set $\gamma_k\cap D$ has a component intersecting both $a$ and $b$, denoted as $\alpha_k$, which is necessarily a closed sub-arc of $\gamma_k$. See Figure \ref{2.6}.
\begin{figure}[ht]
\begin{tikzpicture}[x=1cm,y=1.0cm,scale=0.33]
\fill[thick,gray!10]  (5,0) circle (3);
\draw  (5,0) circle (3);
\draw[thick]  (0,0) circle (5);
\draw[black] (5,3) -- (5.9,1.5) --(5.9,-1.5)-- (5,-3);
\node at (6.6,0) {$\alpha_k$};

\draw  [fill,gray](5,0)circle (0.2);
\node at (4.0,0) {$z_0$};
\node at (9.0,0) {$w_0$};


\node at (-6,0) {$\Gamma_p$};
\node at (7,2.9) {$b$};
\node at (7,-2.8) {$a$};

\draw [black,very thick] (8,0) arc(0:20:3);
\draw [black,very thick] (8,0) arc(0:-20:3);

\end{tikzpicture}
\vskip -0.25cm
\caption{Relative locations of $z_0, w_0, a, b$ and $\alpha_k$.}\label{2.6}
\end{figure}

Let $W_k$ be the union of all these $M_n(1\le n\le N)$ that intersect $\alpha_k$.
Then $W_k$ is connected hence is a domain, that contains the whole arc $\alpha_k$. Since there are finitely many choices for the domains $M_n$, we can find an infinite subsequence, say $\{k_i: i\ge1\}$, such that these domains $W_{k_i}$ coincide with one  another.

We {\bf claim} that each of these domains $W_{k_i}$ contains $W$.

Recall that $W=\bigcup_{n=1}^{N_0}M_n$ and that $z_0\notin\overline{M_n}$ for $ n> N_0$.
Also recall that $U\setminus W$ is contained in the union of those $M_n$, with $N_0<n\le N$ and that $z_0\in\overline{M_n}$ for $1\le n\le N_0$,  we shall have $z_0\notin \overline{U\setminus W}$.
By the previous claim, for  any  $r$  smaller than the distance from $z_0$ to $\overline{U\setminus W}$, the intersection $D_r(z_0)\cap U$ is contained in $W$, which is further contained in $W_{k_1}\subset U$. If $r$ is small enough, with $D_r(z_0)\subset D$, then $N_r:=D_r(z_0)\cup W_{k_1}$ is an open neighborhood of $z_0$ such that
$N_r\cap U=W_{k_1}\cap U=W_{k_1}$.  This will complete our proof, since $W_{k_1}$ is known to be connected.


To verify the above mentioned claim, we connect $z_0$ to a point $w_n\in M_n$ by an open arc $\beta_n\subset M_n$ for $1\le n\le N_0$. Such an arc always exists, since $M_n$ has Property S and since $z_0\in\overline{M_n}$ \cite[p.  111 (4.1)]{Whyburn42}. Recall that $\lim\limits_{i\rightarrow\infty}\alpha_{k_i}=D\cap\Gamma_p$ under Hasudorff distance. This limit is a circular arc, with end points $u$ and $v$. Given $1\le n\le N_0$ and  $\delta>0$ smaller than the distance from $\{u,v\}$ to $\beta_n$, the union $\alpha_{k_i}\cup D_\delta(u)\cup D_\delta(v)$ separates $z_0$ from $w_n$ in $U$ for all but finitely many $i$. This indicates that $\beta_n$ and hence $M_n$ intersects $\alpha_{k_i}$ for all but finitely many $i$. Thus $M_n$ is contained in all but finitely many $W_{k_i}$, indicating that  $W=\bigcup_{n=1}^{N_0}M_n$   is contained in all but finitely many $W_{k_i}$. Consequently, we have $W\subset W_{k_i}$ for $i\ge1$.
\end{proof}

\begin{rema}\label{N_Younsi}
Notice that Ntalampekos and Younsi have obtained the result of Theorem \ref{connected_cluster} in \cite{Ntalampekos-Younsi20}. Observe that, although these authors focus on homeomorphisms  $h: U\rightarrow X$  between planar domains, their proof for \cite[Proposition 3.5]{Ntalampekos-Younsi20} is still valid when $X$ is a general compact metric space and when $h$ is only assumed to be a continuous map. Thus the implication $(1)\Rightarrow(2)$ of Theorem \ref{characterizing_GJD} is already included in \cite{Ntalampekos-Younsi20}. For further details concerning this observation, we refer to \cite[Theorem 3.6]{Ntalampekos-Younsi20} and \cite[Lemma 3.7]{Ntalampekos-Younsi20}.
However, our proof for Theorem \ref{connected_cluster} is based on more direct arguments. These arguments are not dependent on Moore's decomposition theorem \cite{Moore25}, which plays an important role in the proof for \cite[Proposition 3.5]{Ntalampekos-Younsi20}. For the moment, it is not clear whether our strategy works for similar questions  concerning the connectedness of cluster sets for  continuous maps of domains  $U\subset\mathbb{R}^n$ with $n\ge3$ into a  compact metric space. Note that, in such a situation, we do not have a result like Moore's decomposition theorem.
\end{rema}

Now we continue to prove Theorem \ref{connected_cluster_GJD}, by obtaining Propositions \ref{non_PC} and \ref{cut_point}. In the proof for Proposition \ref{cut_point}, we  will need the following lemma whose result is immediate.

\begin{lemm}\label{cluster_preimage}
Let $f: U\rightarrow\Omega$ be a homeomorphism between planar domains such that the inverse $\varphi=f^{-1}: \Omega\rightarrow U$ has a continuous extension to $\overline{\Omega}$, still denoted as $\varphi$. Then $C(f,z_0)=\varphi^{-1}(z_0)$ for all $z_0\in\partial U$.
\end{lemm}

\begin{prop}\label{non_PC}
Let $U\subset\widehat{\mathbb{C}}$ be a domain. If $\partial U$ is not a Peano compactum there is a bounded continuous map
$f: U\rightarrow\mathbb{R}$  such that  $C(f,z_0)$ is disconnected for some $z_0\in \partial U$.
\end{prop}
\begin{proof}
We may assume that $\infty\in\widehat{\mathbb{C}}$ lies in $U$ and consider  $\partial U$ as a compact subset of $\mathbb{C}$. Since $\partial U$ is not a Peano compactum, we may apply \cite[Theorem 4.8]{FLY20} and infer that $K=\widehat{\mathbb{C}}\setminus U$ is not a Peano compactum, either.  Applying  \cite[Theorem 3]{LLY-2019} we see that $K$ does not satisfy the Sch\"onflies condition. See also Definition \ref{Schonflies_condition}. That is to say, there exist two parallel lines $L_1,L_2$  such that the {\bf unbounded closed strip} $W$,  with $\partial W=L_1\cup L_2$, is such that $W\setminus K=W\cap U$ has infinitely many components, say $W_n$ with $n\ge1$, intersecting both $L_1$ and $L_2$.

With no loss of generality we assume that the lines $L_i$ are horizontal.
Moreover,  by scaling we may assume that the distance between $L_1$ and $L_2$ is exactly one. The compactness of $\partial U$ implies that there is a rectangle $R_0\subset W$  containing all but two of those $W_n$. By applying a map of the form $x+\mathbf{i}y\mapsto ax+\mathbf{i}y$, we also assume that $R_0$ is a square of side length one.

Now, by going to an appropriate subsequence, we may further assume that the limit under Hausdorff distance $\lim\limits_{n\rightarrow\infty}\overline{W_n}=M$ exists. This limit $M$ is a continuum. If $z\in M$ then any  disk $D_z$ centered at $z$ intersects infinitely many $W_n$, say, at a point $z_n$. Since all those sets $W_n$ are distinct components of $W\setminus K=W\cap U$,  we can further infer that the segment from $z$ to $z_n$ hence the disk $D_z$ intersects $\overline{W_k}\setminus W_k$ for infinitely many $k$. Therefore, $M$ is a subset of $\partial U$.
Now, fix $z_0\in M$ whose distance to $L_1$ is $\frac12$. We will construct a continuous map $f: U\rightarrow[-\sqrt{2},\sqrt{2}]$ such that the cluster set $C(f,z_0)$ is a subset of $\left[-\sqrt{2},-\frac13\right]\cup\{0\}\cup\left[\frac13,\sqrt{2}\right]$ that intersects the two intervals $\left[-\sqrt{2},-\frac13\right]$ and $\left[\frac13,\sqrt{2}\right]$ both. This then indicates that $C(f,z_0)$ is disconnected.

To do that, we start from a function $f_0: U\rightarrow\mathbb{R}$ defined as follows. Firstly,  $f_0(z)=0$ for all $z\in U$ that do not belong to any $W_n$; secondly, if $z\in W_n$ for some $n\ge1$, let $f_0(z)$ be the infimum of the diameter of all arcs that connect $z$ to a point on $(L_1\cup L_2)\cap W_n$. Clearly, $f_0$ is continuous and its range lies in $[0,\sqrt{2}]$. Now, it suffices to set
$f(z)=f_0(z)$ for $z\notin\left(\bigcup_nW_n\right)$ and $f(z)=(-1)^nf_0(z)$ for $z\in W_{n}$.
Indeed, for this map $f$ and the point $z_0\in M$ chosen as above, the image $f(U\cap D_r(z_0))$ with $r<\frac16$ is a subset of $\left[-\sqrt{2},-\frac13\right]\cup\{0\}\cup\left[\frac13,\sqrt{2}\right]$ that intersects the two intervals $\left[-\sqrt{2},-\frac13\right]$ and $\left[\frac13,\sqrt{2}\right]$ both.
\end{proof}

\begin{prop}\label{cut_point}
Given a domain $U\subset\widehat{\mathbb{C}}$. If $\partial U$ is a Peano compactum but $U$ is not a generalized Jordan domain, there is a continuous map
$f: U\rightarrow\mathbb{D}$  such that the cluster set $C(f,z_0)$ for some $z_0\in\partial U$ is disconnected.
\end{prop}
\begin{proof}
Assume that $\infty\in\widehat{\mathbb{C}}$ lies in $U$.
As $\partial U$ is not a generalized Jordan domain, by Torhorst Theorem \cite[p.  512, \S61, II, Theorem 4]{Kuratowski68}, at least one of the components of $\partial U$, say $Q$, contains a cut point $z_0$.
Here a point  $z_0$ of a connected set $Q$ is called a cut point of $Q$ if $Q\setminus z_0$ is no longer connected, so that there is a separation $Q\setminus\{z_0\}=Q_1\cup Q_2$, satisfying $\overline{Q_1}\cap Q_2=\emptyset=Q_1\cap \overline{Q_2}$.

As $\partial U$ is a Peano compactum,  $Q$ is a Peano continuum.
Let $W_Q$ be the component of $\widehat{\mathbb{C}}\setminus Q$ containing $U$. As $W_Q$ is simply connected, we can choose a Riemann mapping $f_0: W_Q\rightarrow\mathbb{D}$ that sends $\infty$ to $0$. Due to the Carath\'eodory's Continuity Theorem, the inverse map $f_0^{-1}: \mathbb{D}\rightarrow W_Q$ has a continuous extension to the closure $\overline{\mathbb{D}}$, denoted as $\varphi_0$.

Finally, we set $f=\left.f_0\right|_U$ and $\varphi=\left.\varphi_0\right|_{\overline{f_0(U)}}$. By Lemma \ref{cluster_preimage} we have $\varphi^{-1}(z_0)=C(f,z_0)$. So we just need to show that $\varphi^{-1}(z_0)$ is disconnected. Indeed, if on the contrary $\varphi^{-1}(z_0)$  were connected it would be a sub-interval $I$, satisfying $\varphi(\partial\mathbb{D}\setminus I)=Q_1\cup Q_2$. This would be absurd, since $Q\setminus\{z_0\}=Q_1\cup Q_2$ is known to be a separation. \end{proof}

\subsection{On Domains $\Omega\subset\hat{\mathbb{C}}$ Having a  Peano Compactum as Boundary}\label{topology}

In order to obtain Theorem \ref{topology_metric}, we shall prove that the following six conditions are equivalent for all domains $\Omega\subset\hat{\mathbb{C}}$:
\begin{enumerate}
\item[(1)]  $\partial\Omega$ is a Peano compactum.
\item[(2)]  $\Omega$ has Property S.
\item[(3)]   all points of $\partial\Omega$ are locally accessible.
\item[(4)]  all points of $\partial\Omega$ are locally sequentially accessible.
\item[(5)]  $\Omega$ is finitely connected at the boundary.
\item[(6)] The completion $\overline{\Omega}_d$ of $\Omega$ under the Mazurkiewicz distance $d$ is compact.
\end{enumerate}
Our arguments will center around two groups of implications: $(1)\Leftrightarrow(2)\Leftrightarrow(5)\Leftrightarrow(6)$ and $(2)\Rightarrow(3)/(4)$ and $(3)/(4)\Rightarrow(1)$. The equivalence $(1)\Leftrightarrow(2)$ is obtained by {\bf Theorem \ref{property_s}}. The equivalence $(5)\Leftrightarrow(6)$ has been given in \cite[Theorem 1.1]{BBS16}.  The equivalence $(2)\Leftrightarrow(5)$ is to be established in {\bf Theorem \ref{S_finitely_connected}}. The implication $(2)\Rightarrow(3)$ is already known \cite[p.  111, (a)]{Whyburn42} and the implications $(3)/(4)\Rightarrow(1)$ and $(2)\Rightarrow(4)$ will be discussed in {\bf Theorem \ref{3-1-4-1}}.

\begin{rema}\label{notions_for_th.1}
There are four issues we want to mention.
Firstly, {\bf local accessibility} is a synonym for {\bf regular accessibility}, as introduced in \cite[p.  112, Theorem (4.2)]{Whyburn42}. See also Theorem \ref{whyburn_112}.
Secondly, a point $\xi\in\partial\Omega$ is called {\bf locally sequentially accessible} if for each $r>0$ and for each sequence $\{\xi_n\}$ of points in $\Omega$ that converge to $\xi$ the common part $\Omega\cap D_r(\xi)$ is an open set such that one of its components contains infinitely many $\xi_n$.
Thirdly, a domain $\Omega\subset\hat{\mathbb{C}}$ is {\bf finitely connected at a boundary point $x\in\partial\Omega$} provided that for any number $r>0$ there is an open subset $U_x$ of $\hat{\mathbb{C}}$, with $x\in U_x\subset D_r(x)$, such that $U_x\cap\Omega$ has finitely many components. If $\Omega$ is finitely connected at every of its boundary points, we say that $\Omega$ is {\bf finitely connected at the boundary}.  Recall that if $U_x\cap\Omega$ is further required to be connected we say that $\Omega$ is simply connected at $x$ and that if $\Omega$ is simply connected at every of its boundary points we say that $\Omega$ is simply connected at the boundary. See {\bf Theorem \ref{characterizing_GJD}} that characterizes all generalized Jordan domains.
Lastly, the Mazurkiewicz distance $d$ between $x\ne y \in \Omega$ is defined by $d(x,y):=\inf \rm{diam}\ E$, where the infimum is taken over all connected sets $E\subset\Omega$ containing $x$ and $y$ simultaneously. See for instance  \cite{BBS16}.
\end{rema}

\begin{theo}\label{S_finitely_connected}
A domain $\Omega\subset\hat{\mathbb{C}}$ has Property S if and only if it is finitely connected at the boundary.
\end{theo}
\begin{proof} Suppose that $\Omega$ is finitely connected at the boundary. Then, given $r>0$ we can find for any $x\in\partial\Omega$ an open set $G_x$, lying in the open disk centered ar $x$ with radius $\frac{r}{2}$, such that $G_x\cap\Omega$ has finitely many components \cite[Definition 2.2]{BBS16}.
Clearly, the collection $\{G_x: x\in\partial\Omega\}$ gives an open cover of the boundary $\partial\Omega$. So we can find a finite subcover of $\partial \Omega$, denoted as $\{G_1,\ldots, G_n\}$. Since $\Omega\setminus\left(\bigcup G_i\right)$ is a compact subset of $\Omega$, we can cover it with finitely many small disks contained in $\Omega$, with radius $<\frac{r}{2}$. For $1\le i\le n$ the intersection $G_i\cap\Omega$ has finitely many components, each of which is a domain. These domains and the above-mentioned small disks, that cover $\Omega\setminus\left(\bigcup G_i\right)$, form a finite cover of $\Omega$ by subdomains of $\Omega$ having a diameter $<r$. This shows that $\Omega$ has Property S.

In the sequel, we assume that $\Omega$ has Property S. Then, for any $r>0$ we can cover $\Omega$ by finitely many domains $W_1,\ldots, W_N\subset\Omega$ which have a diameter smaller than an arbitrary $\varepsilon\in(0,\frac{r}{3})$. For any $x\in\partial\Omega$, denote by $U_x$ the union of all those $W_i$ whose closure contains $x$ and by $E_x$ the union of all those $W_i$ whose closure does not contain $x$. Then $\overline{E_x}$ is a compact set, whose distance to $x$ is a positive number $r_x>0$.

Let $D$ be the disk centered at $x$ with radius smaller than  $\min\left\{\frac{r}{3},r_x\right\}$ so that $D\cap \Omega\subset U_x$. It is clear that  $G_x=U_x\cup D$ is an open subset of $\widehat{\mathbb{C}}$, containing $x$ and contained in $D_r(x)$, that satisfies $G_x\cap\Omega=U_x$. Since $U_x$ is the union of some of the domains $W_1,\ldots, W_N$, we see that it has finitely many components. This verifies that $\Omega$ is finitely connected at $x$.
\end{proof}

\begin{theo}\label{3-1-4-1}
Theorem \ref{topology_metric} is true, since the implications $(3)/(4)\Rightarrow(1)$ and $(2)\Rightarrow(4)$ hold.
\end{theo}
\begin{proof}
Without losing generality, we may assume that  $\infty\in\Omega$. Under this context $\partial\Omega$ may be considered as a compactum on $\mathbb{C}$.
And we shall  start from the implications $(3)\Rightarrow(1)$ and $(4)\Rightarrow(1)$, each of which will be obtained by a contrapositive proof.

Suppose on the contrary that $\partial\Omega$ were not a Peano compactum. Then it would not satisfy the Sch\"onflies condition \cite[Theorem 3]{LLY-2019}. In other words, there would exist an unbounded closed strip $W$, whose boundary consists of two parallel lines $L_1\ne L_2$, such that $W\setminus\partial\Omega$ has infinitely many components, say $W_n$ for $n\ge1$, each of which intersects both $L_1$ and $L_2$. 

The following construction is similar to that in the proof of Theorem \ref{property_s}. Since every $W_n$ is arcwise connected, we may choose an open arc $\alpha_n\subset W_n$ joining a point on $W_n\cap L_1$ to a point on $W_n\cap L_2$. Suppose that the arcs $\alpha_n$ are arranged inside $W$ linearly from left to right.  See Figure \ref{long-band} for a simplified depiction.
Let $D_n(n\geq1)$ be the bounded component of $\mathbb{C}\setminus(L_1\cup L_2\cup \alpha_n\cup\alpha_{n+1})$. Then each $D_n$ is a Jordan domain; moreover, the closure $\overline{D_n}$ contains a continuum $M_n\subset\partial\Omega$ that separates $\alpha_n$ from $\alpha_{n+1}$ in $\overline{D_n}$ \cite[lemma 3.6]{LLY-2019}. Such a continuum $M_n$ must intersect both $L_1$ and $L_2$.

Now let $L$ be the line parallel to $L_1$ with ${\rm dist}(L,L_1)={\rm dist}(L,L_2)$.
Then $L$ intersects $M_n$ for all $n\ge1$. Pick an infinite sequence $\{z_n\}$ of points lying in $M_n\cap L$. By going to an appropriate subsequence, if necessary, we may assume that  $\lim\limits_{n\rightarrow\infty}z_n=z_0\in\partial\Omega$. Pick for all $n\ge1$ a point $\xi_n\in (\Omega\cap D_n)$  such that $\lim\limits_{n\rightarrow\infty}|\xi_n-z_n|=0$.

For every $n\ge1$, the arc connecting $\xi_n$ to $z_0$ intersects $L_1\cup L_2$. All those arcs are necessarily of diameter $\ge{\rm dist}(L,L_1)$.
This indicates that $z_0$ is not locally accessible from $\Omega$ and hence verifies the implication $(3)\Rightarrow(1)$.
On the other hand, we may consider the subsequence $\{\xi_{2n}\}(n\geq 1)$, every two points of which are in distinct components of $W\cap \Omega$. Fix a neighborhood $V_0$ of $z_0$, which entirely lies in the interior of $W$, then each component of $V_0\cap \Omega$ contains at most one point of $\{\xi_{2n}\}$. This indicates that $z_0$ is not locally sequentially accessible from $\Omega$ and verifies the implication $(4)\Rightarrow(1)$.

The rest of our proof is to verify the implication $(2)\Rightarrow(4)$. Let $z_0\in\partial \Omega$ and $\{z_n\}$ be an arbitrary infinite sequence that lies in $\Omega$ with $\lim\limits_{n\rightarrow\infty}z_n=z_0$. Since $\Omega$ has Property S, for any $r>0$, $\Omega$ is the union of finitely many connected sets of diameter less than $r$. Therefore, there exists a connected set containing infinitely many $z_n$ and certainly such that $z_0$ lies in its boundary. Clearly, this connected set is contained in a single component of $\Omega\cap D_r(z_0)$.
\end{proof}

\section{On Homeomorphisms of a Generalized Jordan Domain}\label{proof_1/2}


The target of this section is to establish the following. 

\begin{theo}\label{topological-cct}
A homeomorphism  $\varphi$ of a generalized Jordan domain $U$  onto a domain $\Omega\subset\hat{\mathbb{C}}$ has a continuous extension $\overline{\varphi}: \overline{U}\rightarrow\overline{\Omega}$ if and only if the conditions below are both satisfied.
\begin{itemize}
\item[(1)] The boundary $\partial\Omega$ is a Peano compactum.
\item[(2)] The oscillations of $\varphi$ satisfy $\liminf\limits_{r\rightarrow 0}\sigma_r(z_0)=0$ for all $z_0\in\partial U$.
\end{itemize}
Suppose that the continuous extension  $\overline{\varphi}: \overline{U}\rightarrow\overline{\Omega}$ exists. Then $\overline{\varphi}$ is a monotone map if and only if $\Omega$ is also a generalized Jordan domain. Consequently, if further no arc on $\partial U$ of positive {\red{diameter}} is sent by $\overline{\varphi}$ to a single point of $\partial\Omega$ then $\overline{\varphi}$ is actually a homeomorphism.
\end{theo}

For any $r>0$ and for any $z_0\in\partial U$, we denote by $C_r(z_0)$ the circle centered at $z_0$ with radius $r$ and define the oscillation of $\varphi$ at $C_r(z_0)\cap U$, denoted as $\sigma_r(z_0)$, to be the supremum of the spherical distance from $\varphi(x)$ to $\varphi(y)$ for all $x,y\in C_r(z_0)\cap U$. Also note that a monotone map is a continuous map between compact metric spaces with connected point inverses .

We shall follow a philosophy  employed in  Arsove's Theorem from \cite{Arsove68-a}, which reads as follows.
\begin{theorem*}
Each of the following is necessary and sufficient for a bounded simply connected plane region $\Omega$ to have its boundary parametrizable as a closed curve {\rm(or equivalently, being a Peano continuum)}:
\begin{enumerate}
\item[(1)] all points of $\partial\Omega$ are locally accessible,
\item[(2)] all points of $\partial\Omega$ are locally sequentially accessible,
\item[(3)] some (equivalently, any) Riemann mapping function $\varphi:\mathbb{D}\rightarrow \Omega$ for $\Omega$ can be extended to a continuous mapping of $\overline{\mathbb{D}}$ onto $\overline{\Omega}$.
\end{enumerate}
\end{theorem*}

The first part of Theorem \ref{topological-cct} is known for simply connected domains $\Omega$
\cite[Theorem 1]{Arsove68-a}. And  a topological counterpart for the OTC Theorem is also established. See \cite[Theorem 2]{Arsove68-a}.  In the sequel, we use {\bf Property $S$} instead of the property of {\bf being locally sequentially accessible}. As in earlier works, such as \cite{Arsove67,Arsove68-a}, we also need to estimate from above the oscillations of the homeomorphism $\varphi: U\rightarrow \Omega$, which is now defined on  a generalized Jordan domain $U$.

To do that, we shall employ a {\bf bijection} between the {\bf boundary components} of $U$ and those of $\Omega$. This bijection has been discussed in \cite{He-Schramm93} and \cite[Proposition 3.1]{Ntalampekos-Younsi20}. It associates to any component $Q$ of $\partial U$ a component $P$ of $\partial\Omega$, which actually consists of all the cluster sets $C(\varphi,z_0)$ with $z_0\in Q$. Denote this bijection $\varphi^B(Q)=P$.


The rest of this section is divided into two parts. Our aim is to consruct a complete proof for Theorem \ref{topological-cct}. The first half is given in Theorem \ref{necessary} and the later half in Theorems \ref{boundary_property} and \ref{monotone_bd_map}. All these materials are presented in  two subsections.

\subsection{Topological Counterpart for Theorems \ref{arsove-1}-\ref{arsove-2}}\label{outline}
We prove the first half of {\bf Theorem \ref{topological-cct}}.

To begin with, let us recall a recent result by He and Schramm: {\em each countably connected domain $\Omega\subset\hat{\mathbb{C}}$ is conformally homeomorphic to a circle domain $D$, unique up to M\"obius equivalence} \cite[Theorem 0.1]{He-Schramm93}. Slightly later,  they even prove that any domain $\Omega\subset\hat{\mathbb{C}}$  is conformally equivalent to some circle domain provided that (1) $\partial\Omega$ has
at most countably many components that are {\bf not} geometric circles or single points and  (2) the collection of those components has a countable closure in the space formed by all the components of $\partial\Omega$ \cite{He-Schramm95a}.
However, Koebe's conjecture is still open if $\partial\Omega$ is complicated.
Therefore, we may focus on domains $\Omega$ such that the boundary $\partial\Omega$ is ``simple'' in some sense.

In other words, we would like to limit our discussions to the case when  $\partial\Omega$ does not possess a difficult topology.
To this end, we  examine the necessary conditions for a homeomorphism $\varphi: U\rightarrow \Omega$ to have a continuous extension to the closure $\overline{U}$. At this point, we even do not require the homeomorphism $\varphi:  U\rightarrow \Omega$ to be conformal.

The first half of Theorem \ref{topological-cct} is exactly the following.

\begin{theo}\label{necessary}
A homeomorphism $\varphi$ of a generalized Jordan domain $U$ onto a domain $\Omega\subset\hat{\mathbb{C}}$ continuously extends to $\overline{U}$ if and only if $\partial\Omega$ is a Peano compactum and $\liminf\limits_{r\rightarrow 0}\sigma_r(z_0)=0$ for all $z_0\in\partial U$.
\end{theo}

\begin{proof} Let us start from the ``only if'' part. Assume that $\varphi$ has a continuous extension $\overline{\varphi}: \overline{U}\rightarrow\overline{\Omega}$.
The uniform continuity of $\overline{\varphi}:\ \overline{U}\rightarrow\overline{\Omega}$ indicates that $\liminf\limits_{r\rightarrow 0}\sigma_r(z_0)=0$  for all $z_0\in\partial U$.
So the only thing to be verified is that the boundary $\partial\Omega$ is a Peano compactum.
Since  $U$ is a generalized Jordan domain, all the components of $\partial U$ are each  a point or a Jordan curve. Since each of those  components is mapped by $\overline{\varphi}$  onto a component of $\partial\Omega$, the uniform continuity of $\overline{\varphi}$ then ensures that $\partial\Omega$ is also a Peano compactum.

In the sequel we continue to prove the ``if'' part. And we just need to show that each cluster set $C(\varphi,z_0)$ is a singleton. Suppose on the contrary that the cluster set $C(\varphi,z_0)$ at $z_0\in\partial U$ contains two points, say $w_1\ne w_2$. Then we can find an infinite sequence $z_n\rightarrow z_0$ of distinct points satisfying $\varphi\left(z_{2n-1}\right)\rightarrow w_1$ and $\varphi\left(z_{2n}\right)\rightarrow w_2$.
Since $\partial\Omega$ is a Peano compactum, by Theorem \ref{property_s} we see that $\Omega$ has Property $S$. That is to say, for any $\varepsilon>0$ we can find finitely many connected subsets of $\Omega$, say $N_1,\ldots,N_k$, satisfying $ \bigcup_iN_i=\Omega$ and $\displaystyle\max_{1\le i\le k}{\rm diam}(N_i)<\varepsilon$.

Here we choose  $\varepsilon<\frac13|w_1-w_2|$. Then, there exist two of those connected sets $N_i$, say $N_1$ and $N_2$, such that
(1) $N_1$ contains infinitely many points in $\{\varphi\left(z_{2n-1}\right)\}$ and that (2) $N_2$ contains infinitely many points in $\{\varphi\left(z_{2n}\right)\}$.
Since $z_n\rightarrow z_0$ and since  $\liminf\limits_{r\rightarrow 0}\sigma_r(z_0)=0$, we can further choose a small enough number $r>0$, with $\sigma_r(z_0)<\varepsilon$, such that each of $\displaystyle N_1\cap\{\varphi\left(z_{2n-1}\right)\}$ and $\displaystyle N_2\cap \{\varphi\left(z_{2n}\right)\}$ contains at least one point that is outside of $\varphi(D_r(z_0)\cap U)$ and at least one point inside.
Therefore, by the connectedness of $N_i$, we will  have $N_i\cap\varphi(C_r(z_0)\cap U)\ne\emptyset$ for $i=1,2$.

Let $M$ be the union of the three sets
$\{w_1\}\cup N_1, \ \varphi(C_r(z_0)\cap U)$,  and $\{w_2\}\cup N_2$; each of them is of diameter $\le\varepsilon$.
As $\sigma_r(z_0)$ is defined to be the diameter of $\varphi(C_r(z_0)\cap U)$, we have
$|w_1-w_2|\le{\rm diam}(M)\le 3\varepsilon$. This is absurd, since we have chosen $\varepsilon<\frac13|w_1-w_2|$.
\end{proof}

\subsection{Topological Counterpart for Theorem \ref{OTC-b}}\label{outline-otc}
We prove the second half of Theorem \ref{topological-cct}.

To do that, we firstly investigate into the boundary behaviour of an arbitrary homeomorphism  $\varphi: U\rightarrow\Omega$  of a generalized Jordan domain $U$, which has a continuous extension  $\overline{\varphi}: \overline{U}\rightarrow\overline{\Omega}$ to the whole closure $\overline{U}$. More precisely, for any component $Q$ of $\partial U$, we will obtain the following fundamental property of the restriction $\left.\overline{\varphi}\right|_Q$.


\begin{theo}\label{boundary_property}
The restriction $\left.\overline{\varphi}\right|_Q: Q\rightarrow P=\overline{\varphi}(Q)$ is  non-alternating.
\end{theo}
\begin{rema}
{\rm
A continuous map $f: A\rightarrow B$ is said to be {\bf non-alternating} provided that for no two points $x,y\in B$ does there exist a separation $A\setminus f^{-1}(x)=A_1\cup A_2$ such that $y$ lies in $f(A_1)\cap f(A_2)$ \cite[p.  127, (4.2)]{Whyburn42}. 
Recall that, when $U$ equals $\mathbb{D}^*=\{z\in\hat{\mathbb{C}}: |z|>1\}$, the above property of being non-alternating is useful and has been employed in earlier studies (although without explicitly mentioning this term) on the  Julia set $K$ of a polynomial $f$ with degree $d\ge2$. In such a case, the domain $\Omega$ is chosen to be the attractive basin of $\infty$ of a polynomial $f$, denoted as $U_\infty$. If $K$ is connected there is a conformal homeomorphism $\varphi:\mathbb{D}^*\rightarrow U_\infty$, called the B\"ottcher map, such that $\varphi(z^d)=f(\varphi(z))$ for all $z\in \mathbb{D}^*$. If further $K$ is locally connected there is a continuous extension $\overline{\varphi}: \overline{\mathbb{D}^*}\rightarrow\overline{U_\infty}$, which necessarily satisfies $\overline{\varphi}(\partial\mathbb{D}^*)=K$. Now, define a closed equivalence on $\mathbb{T}=\partial\mathbb{D}^*$ by setting $x\sim y$ if and only if $\overline{\varphi}(x)=\overline{\varphi}(y)$. For any $x\in\mathbb{T}$ let $[x]_\sim$ denote the equivalence class containing $x$.
Moreover, $[x]_\sim$ can be considered as a subset of the closed unit disk $\overline{\mathbb{D}}$. Theorem \ref{boundary_property} says that the restriction $\left.\overline{\varphi}\right|_{\mathbb{T}}$ is non-alternating, which happens if and only if the convex hulls of two arbitrary classes $[x]_\sim \ne [y]_\sim$ in $\overline{\mathbb{D}}$ are disjoint.
Among others, one may refer to \cite[Proposition II.3.3]{Thurston09}. Such an equivalence on $\mathbb{T}$ is also said to be {\bf unlinked}. See for instance \cite{Douady93} or \cite[Definition 4.9]{Kiwi04}.
}
\end{rema}

\begin{proof}[{\bf Proof for Theorem \ref{boundary_property}}]
By {\bf Zoretti's Theorem}, $P$ is a component of $\partial\Omega$. By definition of {\bf non-alternating}, we only need to show that if  $Q\setminus\left(\overline{\varphi}\right)^{-1}\!(x)=A_1\cup A_2$ is a  separation  for some $x\in P$ then $\overline{\varphi}(A_1)\cap \overline{\varphi}(A_2)=\emptyset$.

Assume on the contrary that there were a point $x\in P$ and a separation $Q\setminus\left(\overline{\varphi}\right)^{-1}\!(x)=A_1\cup A_2$ such that $\overline{\varphi}(z_1)=\overline{\varphi}(z_2)$ for $z_i\in A_i (i=1,2)$.
Set $x'=\overline{\varphi}(z_1)$. 
We will show that such an assumption is absurd.

Since $U$ is a generalized Jordan domain and since $Q$ is assumed to be a Jordan curve, the point inverse $\left(\overline{\varphi}\right)^{-1}\!(x)$ contains two points $y_1\ne y_2$ such that $\{y_1,y_2\}$ separates $z_1$ from $z_2$ in $Q$. By Theorem \ref{property_s}, we know that $U$ has property $S$.
Moreover, applying \cite[p.  111, (a)]{Whyburn42}, 
we can further infer that each boundary point of $U$ is accessible from $U$.
Therefore, we can find an open arc $\alpha\subset U$ that connects $y_1$ to $y_2$.
Clearly, the union $Q\cup\alpha$ is a $\theta$-curve \cite[p.  104, Definition]{Whyburn42} and $U\setminus\alpha$ consists of two domains.
Let $U_i (i=1,2)$ be the one whose boundary contains $z_i$. See the left part of Figure \ref{5.1}, for relative locations of $\alpha$, the domains $U_i$ and the points $y_i,z_i$.
\begin{figure}[ht]
\vskip -0.25cm
\begin{tabular}{cccc}
\begin{tikzpicture}[x=1cm,y=1.0cm,scale=0.33]
\draw[thick]  (0,0) circle (5);
\draw[black, very thick, dotted] (3,4) -- (11,4) --(11,-4)-- (3,-4);
\node at (8,-4.5) {$\alpha$};

\draw  [fill,black](3,4)circle (0.2);
\node at (2.2,3.7) {$y_1$};
\draw  [fill,black](3,-4)circle (0.2);
\node at (2.2,-3.7) {$y_2$};
\draw  [fill,gray](5,0)circle (0.2);
\node at (4.0,0) {$z_2$};
\draw  [fill,gray](-5,-0)circle (0.2);
\node at (-4.0,0) {$z_1$};

\node at (8,5) {$U_1$};
\node at (8,0) {$U_2$};
\end{tikzpicture}
&&&
\begin{tikzpicture}[x=1cm,y=1.0cm,scale=0.33]
\draw[thick]  (0,0) circle (5);
\draw[black, very thick, dotted] (3,4) -- (11,4) --(11,-4)-- (3,-4);
\node at (8,-4.5) {$\alpha$};

\draw [thin, gray](8,4)-- (8,2)-- (14,2) -- (14,0)--(11,0);
\node at (15,1) {$\beta$};
\draw [thin, gray] (-5,0) -- (-5,5.4)  -- (8,5.4)-- (8,4);
\node at (2,6.2) {$\beta_1$};
\draw[gray, thin] (8,4) -- (11,4) --(11,0);

\draw [very thick, black] (8,4)--(11,4)--(11,0);
\node at (12,4) {$\beta_3$};
\draw [thin, gray] (11,0)--(5,0);
\node at (8,-1) {$\beta_2$};

\draw  [fill,black](3,4)circle (0.2);
\node at (2.2,3.7) {$y_1$};
\draw  [fill,black](3,-4)circle (0.2);
\node at (2.2,-3.7) {$y_2$};
\draw  [fill,gray](5,0)circle (0.2);
\node at (4.0,0) {$z_2$};
\draw  [fill,gray](-5,-0)circle (0.2);
\node at (-4.0,0) {$z_1$};


\draw  [fill, gray](8,4)circle (0.2);
\node at (8.7,4.7) {$b_1$};
\draw  [fill, gray](11,0)circle (0.2);
\node at (11.7,-0.7) {$b_2$};
\end{tikzpicture}
\end{tabular}
\vskip -0.25cm
\caption{The domains $U_i$, the sub-arcs $\beta_i$ and the points $y_i, z_i, b_i$.}\label{5.1}
\end{figure}

Now, fix an arc $\beta\subset U$ that connects $z_1$ to $z_2$ and denote by $\beta_i\subset(U_i\cap\beta)$  the maximal open sub-arc of $\beta$ that has $z_i$ as one of its ends. Denote by $b_i$ the other end point of $\beta_i$ for $i=1,2$. Obviously,  $b_1,b_2\in\alpha$. Let $\beta_3$ be the closed sub-arc of $\alpha$ between $b_1$ and $b_2$. Then we have an arc $\beta'=\beta_1\cup\beta_2\cup\beta_3$, lying in $U$ and intersecting $\alpha$ at $\beta_3$. See the right part of Figure \ref{5.1}.

Since $\varphi: U\rightarrow\Omega$ is a homeomorphism, we know that $\varphi(\beta')$ 
is an arc contained in $\Omega$ such that (1) $\varphi(\beta')\cap\varphi(\alpha)=\varphi(\beta_3)$ and (2) $\varphi(\beta_i)\subset \varphi(U_i)$ for $i=1,2$. As $\varphi$ is a homeomorphism, we see that $J=\varphi(\alpha)\cup\{x\}$ is a Jordan curve and $\Omega\setminus J=\varphi(U_1)\cup \varphi(U_2)$. Since  $J$ does not contain the point $x'=\overline{\varphi}(z_1)=\overline{\varphi}(z_2)$ and since each of $\varphi(\beta_1)$ and $\varphi(\beta_2)$ has $x'$ as one of its ends, we can infer that $\varphi(\beta_1)$ and $\varphi(\beta_2)$ are both contained in the same component of $\hat{\mathbb{C}}\setminus J$. Therefore, one of the two components of $\Omega\setminus J$,  either $\varphi(U_1)$ or $\varphi(U_2)$, will contain $\varphi(\beta_1)\cup\{x'\}\cup\varphi(\beta_2)$. This is impossible, since we have chosen the arcs $\beta_i\subset U_i(i=1,2)$ in a way so that $\varphi(\beta_i)\subset\varphi(U_i)$.
\end{proof}

A continuous map $f: X\rightarrow Y$ of a continuum onto another is monotone if and only if its point inverses are each a continuum \cite[p. 70]{Whyburn42}. By definition, a monotone map is non-alternating. 
If $X=[0,1]$ then $Y$ is homeomorphic with $[0,1]$ or $\{0\}$ \cite[p. 129, Proposition 8.22]{Nadler92}. Similarly, if $X$ is a simple closed curve and if $Y$ is not a single point then $Y$ is a simple closed curve, too. This follows from the observation that $X\setminus f^{-1}(\{x,y\})$ and hence $Y\setminus\{x,y\}$ are both disconnected for any $x\ne y\in Y$, since by \cite[p. 58, Theorem]{Whyburn42} we readily see that $Y$ is a simple closed curve.
Therefore, if the extension $\overline{\varphi}: \overline{U}\rightarrow\overline{\Omega}$ is monotone the codomain $\Omega$ is also a generalized Jordan domain.
We will show that the converse of this remains true. That is to say, we will have the following Theorem \ref{monotone_bd_map}.

\begin{theo}\label{monotone_bd_map}
Let $\varphi: U\rightarrow\Omega$ be a homeomorphism from a generalized Jordan domains $U$ onto a domain $\Omega\subset\hat{\mathbb{C}}$. Suppose that $\varphi$ has a continuous extension  $\overline{\varphi}: \overline{U}\rightarrow\overline{\Omega}$. Then $\overline{\varphi}$ is  monotone if and only if $\Omega$ is also a generalized Jordan domain.
\end{theo}

\begin{proof}
We just discuss the ``if'' part and show that if a component $P$ of $\partial\Omega$ is a Jordan curve and $Q$ is the component of $\partial U$ with $\overline{\varphi}(Q)=P$ then the restriction $\left.\overline{\varphi}\right|_Q$ is necessarily monotone. And we shall verify that $Q\setminus\overline{\varphi}^{-1}(x)$ is connected for any  $x\in P$, which forces  $\overline{\varphi}^{-1}(x)$ to be a continuum and completes the proof. Indeed, if on the contrary $Q\setminus\overline{\varphi}^{-1}(x)$ has at least two components, say $A_1$ and $A_2$, we may set $B_1=Q\setminus\left(\overline{\varphi}^{-1}(x)\cup A_1\right)$ and verify that $Q\setminus\overline{\varphi}^{-1}(x)=A_1\cup B_1$ is a separation. By Theorem \ref{boundary_property}, the restriction $\left.\overline{\varphi}\right|_Q$ is non-alternating, so that $\overline{\varphi}(A_1)\cap \overline{\varphi}(B_1)=\emptyset$. This indicates that $\overline{\varphi}(A_1)=P\setminus\overline{\varphi}(Q\setminus A_1)$ and $\overline{\varphi}(B_1)=P\setminus\overline{\varphi}(Q\setminus B_1)$ are disjoint and are both open in $P$. Since none of them contains $x$, we see that $P\setminus\{x\}=\overline{\varphi}(A_1)\cup \overline{\varphi}(B_1)$ is a separation. This contradicts the fact that $P\setminus\{x\}$ is a sub-arc of $P$.
\end{proof}

\begin{rema}
Theorem \ref{monotone_bd_map} establishes the second half of Theorem \ref{topological-cct}.
\end{rema}

\section{Generalization of Carath\'eodory's Continuity Theorem}\label{outline-cct}
Our target is to prove Theorems \ref{arsove-1}
to \ref{OTC-b}, which will be done separately in three subsections.

\subsection{The First Generalization of Continuity Theorem}\label{outline-1}
Here we start off by recalling a useful result \cite[Lemma 1.3]{He-Schramm94}, which reads as follows.
\begin{lemm}\label{HS-1994_1.3}
Let $Z\subset\mathbb{R}^2$ be a Borel set of $\sigma$-finite linear measure, and let $X\subset\mathbb{R}$ be the set of points $x$ such that the section $\left(\{x\}\times\mathbb{R}\right)\cap Z$ is uncountable. Then $X$ has zero Lebesgue measure.
\end{lemm}

In the above lemma, we may consider $\mathbb{R}^2$ as the complex plane $\mathbb{C}$, that consists of the points $re^{{\bf i}\theta}$ with $r>0$ and $0\le\theta<2\pi$. Then, we study the set $R_0$ of numbers $r>0$ such that the circle $\left\{re^{{\bf i}\theta}: 0\le\theta<2\pi\right\}$ intersects $Z$ at uncountably many points. For any $r_2>r_1>0$, we see that the part of $Z$ in the annulus $\{z\in\mathbb{C}: r_1\le|z|\le r_2\}$ is sent onto the rectangle $[r_1,r_2]\times[0,2\pi]$ by the map $re^{{\bf i}\theta} \mapsto (r,\theta)$. If we define the distance between $r_1e^{{\bf i}\theta_1}$ and $r_2e^{{\bf i}\theta_2}$ to be $|r_1-r_2|+|\theta_1-\theta_2|$, the previous map and its inverse are both Lipschitz. Therefore, by Lemma \ref{HS-1994_1.3}, we have the following.

\begin{lemm}\label{small_level_sets}
Given a domain $D$ and a point $z_0\in \partial D$. Let $R_0$ denote the set of all $r>0$ such that $C_r(z_0)=\{z: |z-z_0|=r\}$ contains uncountably many point components of $\partial D$. If the boundary components of $\partial D$ form a set with $\sigma$-finite linear measure then $R_0$ has zero Lebesgue measure.
\end{lemm}

We also need to the following.
\begin{lemm}[{\bf Wolff's Lemma, \cite[p.  20, Proposition 2.2]{Pom92}}]\label{Wolff}
Let $\varphi$ map a domain $D\subset\mathbb{C}$ conformally into a disk $D_R(0)\subset\mathbb{C}$. Let $\Lambda_r(z_0)$ be the arc length of $\varphi(C_r(z_0)\cap D)$ with $z_0\in\mathbb{C}$, under Euclidean metric. Then we have
\[
\displaystyle\inf\limits_{\rho<r<\sqrt{\rho}}\Lambda_r(z_0)\le\frac{2\pi R}{\sqrt{\log1/\rho}} \quad(0<\rho<1).
\]
\end{lemm}


The proof of Wolff's Lemma also implies the next.
\begin{lemm}
\label{Wolff_lemma_2}
Let $\varphi: D\rightarrow\Omega$ be a conformal map between domains in $\widehat{\mathbb{C}}$ and let $z_0\in\widehat{\mathbb{C}}$. Then for any $\epsilon,\rho\in(0,1)$
there exists a set of positive measure $E\subset(0,\rho)$ such that $\Lambda_r(z_0)<\epsilon$ for all $r\in E$. In particular, if $R_0\subset(0,\infty)$ is a set of measure zero, then $\liminf\limits_{r\rightarrow0,r\notin R_0}\Lambda_r(z_0)=0$.
\end{lemm}

By Wolff's Lemma, we have $\liminf\limits_{r\rightarrow 0}\Lambda_r(z_0)=0$. This is however different from what we need to verify, which is $\liminf\limits_{r\rightarrow 0}\sigma_r(z_0)=0$. Here  the oscillation $\sigma_r(z_0)$ is defined to be the {\bf diameter} of  $\varphi(C_r(z_0)\cap D)$. Also note that a combination of Theorem \ref{topological-cct} with the following Theorem \ref{oscillation2}  implies {\bf Theorem \ref{arsove-1}}.

\begin{theo}\label{oscillation2}
Let $\varphi: D\rightarrow \Omega\subset\hat{\mathbb{C}}$ be a conformal homeomorphism of a circle domain $D$. Assume that the point components of $\partial D$ form a set of $\sigma$-finite linear measure and $\partial\Omega$ has countably many nondegenerate components $\{P_k\}$. Let $Q_k$ be the component of $\partial D$ with $\varphi^B(Q_k)=P_k$ for all $k\ge1$. If for some fixed $n$ the component $P_n$ has the property of {\bf local diameter control}, so that there is an open set $\Omega_n\supset P_n$ that satisfies  $\displaystyle\sum_{P_k\subset \Omega_n}{\rm diam}(P_k)<\infty$, then $\liminf\limits_{r\rightarrow 0}\sigma_r(z_0)=0$  for all $z_0\in\partial Q_n$.
\end{theo}


\begin{remark*}
In Theorem \ref{oscillation2}, we do not require that $\partial\Omega$ be a Peano compactum. The only assumptions are about the linear measure of the point components of $\partial D$ and about the diameters of $P_k$. Therefore, the result we obtain here is just the oscillation convergence  $\liminf\limits_{r\rightarrow 0}\sigma_r(z_0)=0$  for all $z_0\in\partial Q_n$. Here we say nothing about the cluster sets  $C(\varphi,z_0)$  for $z_0\in\partial Q_n$.
\end{remark*}

\begin{proof}[{\bf Proof for Theorem \ref{oscillation2}}]
For any $r>0$ and any $z_0\in Q_n$, let
$C_r(z_0)=\{|z-z_0|=r\}$. By Lemma \ref{small_level_sets}, the
boundary components of $D$ that intersect $C_r(z_0)$ forms a
countable set for all $r>0$ except for those lying in a set $R_0$ of
zero Lebesgue measure.

Let $\Lambda_r(z_0)$ be the {\bf arc length} of
$\varphi(C_r(z_0)\cap D)$.
By Lemma \ref{Wolff_lemma_2}, we have $\liminf\limits_{r\notin R_0, r\rightarrow
0}\Lambda_r(z_0)=0$. Moreover, we can choose for any $\epsilon>0$ a
decreasing sequence of numbers in $(0,1)\setminus R_0$, say
$r_1>r_2>\cdots\cdots$, such that $\lim\limits_{m\rightarrow
0}r_m=0$ and that $\Lambda_{r_m}(z_0)<\frac12\epsilon$ for all
$m\geq 1$.

Let $\mathcal{Q}_D$ be the family of all those nondegenerate components $Q\ne Q_n$ of $\partial D$ that intersect
$\bigcup_mC_{r_m}(z_0)$. Then $\mathcal{Q}_D$ has at most countably many elements, denoted as $T_i$ with $i=1,2,\cdots$.

Fix a point $w_0\in \Omega$ and use Zoretti's Theorem to find for each $\varphi^B(T_i)$ a Jordan curve $\Gamma_i\subset\Omega$ such that $\Gamma_i$ separates $w_0$ from $\varphi^B(T_i)$ and that every point of $\Gamma_i$ is at a distance less than $2^{-i}\epsilon$ from some point of $\varphi^B(T_i)$. Let $W_i$ denote the component of
$\hat{\mathbb{C}}\setminus \Gamma_i$ that contains $\varphi^B(T_i)$;
moreover, let $W_i^*$ denote the component of
$\hat{\mathbb{C}}\setminus \varphi^{-1}(\Gamma_i)$ that contains
$T_i$. Similarly, we can find for $P_n$ a Jordan domain $W\subset\Omega_n$ with $\Gamma=\partial W\subset \Omega$ such that $\Gamma$ separates $w_0$ from $P_n$. The corresponding Jordan domain $W^*$ is a neighborhood of $Q_n$. We may require that every $\varphi^B(T_i)$ lies in $W$ by choosing a
sufficiently small $r_1$, since $T_i\subset W^*$ if and only if $\varphi^B(T_i)\subset W$.

From the assumption that $\displaystyle\sum_{P_k\subset \Omega_n}{\rm diam}(P_k)<\infty$, we have $\displaystyle\sum_{i}{\rm diam}(\Gamma_i)<\infty$. In the sequel, fix an integer $N\ge1$ with $\displaystyle\sum_{i=N+1}^\infty{\rm diam}(\Gamma_i)<\frac12\epsilon$ and choose
an integer $M$ large enough such that $C_{r_m}(z_0)\setminus Q_n$ is disjoint from $\bigcup_{i=1}^NT_i$ for all $m\ge M$.

Consequently, we only need to verify that the
inequality $|\varphi(z_1)-\varphi(z_2)|<\epsilon$ holds for any
points $z_1\ne z_2$ lying on $C_{r_m}(z_0)\cap D$, where $m\ge M$.
To do that, we may consider the closed sub-arc of
$C_{r_m}(z_0)\setminus Q_n$ from $z_1$ to $z_2$. Denote this arc
as $\alpha$, which is a compact set disjoint from each of
$Q_n, T_1,\ldots, T_N$. The components of $\alpha\cap D$ form a
countable family $\{\alpha_t: t\in\mathcal{I}\}$. All these
$\alpha_t$ are open arcs or semi-closed arcs on the circle
$C_r(z_0)$.
Since
\[\left\{W_i^*: i\ge N+1\right\}\ \bigcup\
\left\{\alpha_t: t\in\mathcal{I}\right\}
\]
is a cover of $\alpha$, a compact set, and since each $\alpha_t$ is open in $\alpha$, we may choose finite index sets $\mathcal{J}_0\subset\{N+1,N+2,\ldots\}$ and
$\mathcal{I}_0\subset\mathcal{I}$ such that
\[\mathcal{C}_\alpha^*:=\left\{W_i^*: i\in\mathcal{J}_0\right\}\ \bigcup\
\left\{\alpha_t: t\in\mathcal{I}_0\right\}
\]
is  a finite cover of $\alpha$ and $W_i^*\cap \alpha\neq\emptyset$ for all $i\in\mathcal{J}_0$.
Choose some elements of $\mathcal{C}_\alpha^*$, say $C_1,\ldots,C_n$  with $z_1\in C_1$ and $z_2\in C_n$, that form a chain from $z_1$ to $z_2$ in the sense that $C_i\cap C_{i+1}\ne\emptyset$ for $1\le i\le n-1$.
Notice that  $\mathcal{C}_\alpha:=\{W_i: i\in\mathcal{J}_0\} \cup \{\varphi(\alpha_t): t\in\mathcal{I}_0\}$ consists of finitely many connected sets, by which we can construct a chain from $\varphi(z_1)$ to $\varphi(z_2)$.
Then we can infer  that
\begin{equation}\label{inequality-a}
|\varphi(z_1)-\varphi(z_2)|<\sum_{i\in\mathcal{J}_0}{\rm
diam}(\Gamma_i)+\sum_{t\in\mathcal{I}_0}{\rm
diam}(\varphi(\alpha_t))<\frac12\epsilon+\frac12\epsilon=\epsilon.
\end{equation}
By the flexibility of $z_1,z_2\in C_{r_m}(z_0)\cap
D$, this completes our proof.
\end{proof}

We may assume that the circle domain $D$ in {\bf Theorem \ref{arsove-1}} is just a generalized Jordan domain and keep the other assumptions. In other words, we shall have the following theorem.

\begin{theo}\label{arsove_sigma-new}
Let $\varphi: \Omega_1\rightarrow \Omega_2\subset\hat{\mathbb{C}}$ be a conformal homeomorphism of  a generalized Jordan domain  $\Omega_1$. Suppose that the point components of $\partial\Omega_1$ form a set with $\sigma$-finite linear measure and $\partial\Omega_2$ has at most countably many nondegenerate components $\{P_n\}$ with $\displaystyle\sum_n{\rm diam}(P_n)<\infty$ .  Then $\varphi$  extends continuously to the closure $\overline{\Omega_1}$ if and only if $\partial\Omega_2$ is a Peano compactum.
\end{theo}
\begin{proof}
Let $Q$ be the component of $\partial \Omega_1$ with $\varphi^B(Q)=P$. If $Q$ is a point component of $\partial \Omega_1$, the above proof for Theorem \ref{oscillation2} still works. Thus by Theorem \ref{topological-cct}, the conformal homeomorphism $\varphi: \Omega_1\rightarrow\Omega_2$ continuously extends over $Q$, which indicates that $P$ is necessarily a single point.

If $Q$ is a Jordan curve and $U_0$ is the component of $\hat{\mathbb{C}}\setminus Q$ containing $\Omega_1$, we can find a homeomorphism $H:\hat{\mathbb{C}}\rightarrow \hat{\mathbb{C}}$, sending $Q$ onto the unit circle $\partial\mathbb{D}$,  such that $\left.H\right|_{U_0}$ is a conformal map between $U_0$ and
$\{z\in\hat{\mathbb{C}}: |z|>1\}$. Notice that $\varphi: \Omega_1\rightarrow \Omega_2$ continuously extends over $Q$ if and only if $\varphi\circ H^{-1}: H(\Omega_1)\rightarrow\Omega_2$ continuously extends over $\partial\mathbb{D}$. Denote by $X$ the set consisting of all the point components of $\partial \Omega_1$. Since $H(X)$ consists of all the point components of $\partial H(\Omega_1)$ and hence it is also of $\sigma$-finite linear measure, the above proof for Theorem  \ref{oscillation2} is still valid and we can infer that the oscillation of $\varphi\circ H^{-1}$ at every $w\in\partial\mathbb{D}$ has zero as its inferior limit,
as $r\rightarrow 0$. With this, we may apply Theorem \ref{topological-cct} and infer that   $\varphi\circ H^{-1}: H(\Omega_1)\rightarrow\Omega_2$ continuously extends over $\partial\mathbb{D}=H(Q)$, indicating that $\varphi: \Omega_1\rightarrow\Omega_2$ continuously extends over $Q$.
\end{proof}

\subsection{The Second Generalization of Continuity Theorem}\label{outline-2}
We give a complete proof for  Theorem \ref{arsove-2}.
And the only issue is to verify that $\liminf\limits_{r\rightarrow 0}\sigma_r(z_0)=0$ holds for all $z_0\in\partial D$.
In other words, we need to obtain the following.

\begin{theo}\label{oscillation-a}
Let $\varphi: D\rightarrow \Omega\subset\hat{\mathbb{C}}$ be a conformal homeomorphism of a circle domain $D$. Assume that $\partial\Omega$ has countably many non-degenerate components $\{P_k\}$ and all its point components form a set of $\sigma$-finite linear measure.
Let $Q_k$ be the component of $\partial D$ with $\varphi^B(Q_k)=P_k$ for all $k\ge1$. If for some fixed $n$, $P_n$ has the property of local diameter control, so that there exists an open set $\Omega_n\supset P_n$ satisfying  $\displaystyle\sum_{P_k\subset \Omega_n}{\rm diam}(P_k)<\infty$, then $\liminf\limits_{r\rightarrow 0}\sigma_r(z_0)=0$  for all $z_0\in\partial Q_n$.
\end{theo}

We can infer Theorem \ref{arsove-2} by combining Theorem \ref{topological-cct} and the above Theorem \ref{oscillation-a}. Before proving Theorem \ref{oscillation-a}, we need the following key lemma.

\begin{lemm}\label{zero linear measure}
Let $F_r$ consist of all points $q\in\partial\Omega$ such that $\{q\}=\varphi^B(Q)$ for some component $Q$ of $\partial D$ that intersects $C_{r}(z_0)$. Suppose that all the point components of $\partial \Omega$ constitute a set of $\sigma$-finite linear measure, denoted as $B_\Omega$. Let $E\subset (0, \infty)$. Then the linear measure of $F_r$ is zero for all but countably many  $r\in E$.
\end{lemm}

\begin{proof}
Let $F$ consist of all the points $q\in\partial\Omega$ such that $\{q\}=\varphi^B(Q)$ for some non-degenerate component $Q$ of $\partial D$. As $F$ is at most countable and the sets $\widetilde{F_r}=F_r\setminus F$ are pairwise disjoint, we only need to verify that the linear measure of $\widetilde{F_r}$ is zero for all but countably many  $r\in E$.

To do that, we consider two closed maps. One is the natural projection $\pi_1$ of the upper semi-continuous decomposition $\mathcal{D}_1$ of $\overline{\Omega}$, whose elements are either singletons lying in $\Omega$ or boundary components of $\Omega$. The other one, denoted as $\pi_2$, sends every point $z\in \overline{D}$ to $\pi_1(\varphi(z))$ if $z\in D$ or to $\pi_1\left(\varphi^B(Q)\right)$ if $z$ lies on a component $Q$ of $\partial D$. As $\pi_2\left(C_r(z_0)\cap\partial D\right)$ is closed in the quotient space $\mathcal{D}_1$, its preimage under $\pi_1$ is a closed subset of $\overline{\Omega}$, denoted as $B_r$. Since  $F_r\subset B_r$ and since $B_r\setminus F_r$ consists of at most countably many components of $\partial\Omega$,  every $F_r$ and hence every $\widetilde{F_r}$ is a Borel set. As $B_\Omega$ is assumed to have $\sigma$-finite linear measure, by Borel regularity of  the linear measure $\mathcal{H}^1$ \cite[p. 61, Theorem 1]{Evans_Gariepy} we can write $B_\Omega$ as the union of at most countably many Borel sets $A_i$  with finite linear measure.
Now we claim that   $\mathcal{H}^1\left(\widetilde{F_r}\right)=0$ for all but countably many  $r\in E$. Otherwise,  there would exist some $i\geq 1$ such that $\mathcal{H}^1\left(\widetilde{F_r}\cap A_i\right)>0$ for uncountably many $r$.  This is impossible, since $\widetilde{F_r}\cap \widetilde{F_s}=\emptyset$ for $s\neq r$ and $\mathcal{H}^1(A_i)<\infty$.
\end{proof}

Now we are well prepared to construct a proof for Theorem \ref{oscillation-a}.

\begin{proof}[{\bf Proof for Theorem \ref{oscillation-a}}]
Let $\{k_i: i\ge1\}$ be the collection of all those  $k_i$ with $P_{k_i}\subset \Omega_n$, arranged  so that $k_1<k_2<\cdots$. Here we also recall that $Q_{k_i}$ denotes the component of $\partial D$ with $P_{k_i}=\varphi^B\left(Q_{k_i}\right)$.

Given a point  $z_0\in\partial Q_n$ and an arbitrary number $\epsilon>0$, we shall find a positive number $r<\epsilon$ such that $\sigma_r(z_0)<\epsilon$, which then completes our proof.

To this end, we firstly fix a point $w_0\in \Omega$ and then use
Zoretti's Theorem to find a simple closed curve $\Gamma_i\subset \Omega$ for each $P_{k_i}$ such that $\Gamma_i$ separates $w_0$ from $P_{k_i}$ and that every point of $\Gamma_{k_i}$ is at a distance less than $2^{-i}\epsilon$ from some point of $P_{k_i}$.
Clearly, we have $\sum_{i}{\rm diam}(\Gamma_i)<\infty$. And there is  an integer $N\ge1$ with
$\sum_{i=N+1}^\infty{\rm diam}(\Gamma_i)<\frac12\epsilon.$
In the sequel, for $i\ge1$, let $W_i$ be the component of $\hat{\mathbb{C}}\setminus \Gamma_i$ that contains $P_{k_i}$ and
$W_i^*$ the component of $\hat{\mathbb{C}}\setminus \varphi^{-1}(\Gamma_i)$ that contains $Q_{k_i}$.

By Lemma \ref{Wolff_lemma_2}, we can find $E\subset(0, \epsilon)$ of positive measure such that for all $r\in E$ we have the next three properties: (1) $\displaystyle\Lambda_r(z_0)<\frac14\epsilon$, (2)
$\varphi\left(C_r(z_0)\cap D\right)\subset \Omega_n$, (3) $C_r(z_0)\setminus Q_n$ intersects none of the  components $Q_{k_1},\ldots, Q_{k_N}$ of $\partial D$ for all $r\in E$.
By Lemma \ref{zero linear measure}, the linear measure $\mathcal{H}^1(F_r)$ is zero for some $r\in E$. Thus for the above $\epsilon>0$ we can choose open sets $V_k$, with $k$ running through some index set $\mathcal{K}\subset\mathbb{N}$ that is finite or countably infinite, such that $\{V_k: k\in\mathcal{K}\}$ is an open cover of $F_r$ satisfying $\displaystyle \sum_{k\in\mathcal{K}}{\rm diam}\left(V_k\right)<\frac14\epsilon$. Here all those $V_k$ may be required to have a diameter smaller than an arbitrary constant $\delta>0$.

Since every $V_k$ is open in $\hat{\mathbb{C}}$, we can use Zoretti's Theorem to find  for each $q\in F_r$ a Jordan curve $J_{q}\subset\Omega$, lying in some $V_k$, that separates $w_0$ from the point component $\{q\}$ of $\partial\Omega$. Let $U_q$ be the component of $\hat{\mathbb{C}}\setminus J_q$ that contains $q$. Let $U_q^*$ be the component of $\hat{\mathbb{C}}\setminus\varphi^{-1}(J_q)$   containing  the component of $\partial D$ that is sent to $\{q\}$ by $\varphi^B$.
By the flexibility of $\epsilon$, we only need to verify that for the above mentioned $r$, the inequality $|\varphi(z_1)-\varphi(z_2)|<\epsilon$ holds for any fixed points $z_1\ne z_2$ lying on $C_{r}(z_0)\cap D$.

To this end, we shall consider the closed sub-arc of $C_{r}(z_0)\setminus Q_n$ that connects $z_1$ to $z_2$. This arc, denoted as $\alpha$, is a compact set disjoint from each of $Q_n, Q_{k_1},\ldots, Q_{k_N}$. Moreover, the components of $\alpha\cap D$ form a countable family $\{\alpha_t: t\in\mathcal{I}\}$. Let $\mathcal{J}=\{i: i\ge N+1\}$.
Then
\begin{equation}\label{cover_C0}
\left\{W_i^*: i\in\mathcal{J}\right\}\ \bigcup\
\left\{U_q^*: q\in F_r\right\}\ \bigcup\
\left\{\alpha_t: t\in\mathcal{I}\right\}
\end{equation}
is a cover of $\alpha$. Since each $\alpha_t$ is open in $\alpha$, the above collection  is an open cover of $\alpha$. By the compactness of $\alpha$, we may choose finite index sets $\mathcal{J}_0\subset\mathcal{J}$, $F_0\subset F_r$ and $\mathcal{I}_0\subset\mathcal{I}$, such that
\[\mathcal{C}_\alpha^*:=\left\{W_i^*: i\in\mathcal{J}_0\right\}\ \bigcup\
\left\{U_q^*: q\in F_0\right\}\ \bigcup\
\left\{\alpha_t: t\in\mathcal{I}_0\right\}
\]
is  a finite cover of $\alpha$.
As in Theorem \ref{oscillation2}, using elements of $\mathcal{C}_\alpha^*$ we can find a finite chain from $z_1$ to $z_2$.
Similarly, using the elements of $\mathcal{C}_\alpha:=\{W_i: i\in\mathcal{J}_0\} \ \cup\
\left\{U_q: q\in F_0\right\}\ \cup\
\left\{\varphi(\alpha_t): t\in\mathcal{I}_0\right\}$ we can construct a finite chain from $\varphi(z_1)$ to $\varphi(z_2)$. As every $U_q$ lies in some $V_k$,  we can further find a finite index $\mathcal{K}_0 \subset \mathcal{K}$ such that every $U_q$ with $q\in F_0$ is contained in $V_k$ for some $k\in\mathcal{K}_0$. Now we can use the elements of $\{W_i: i\in\mathcal{J}_0\} \ \cup\
\left\{V_k: k\in \mathcal{K}_0\right\}\ \cup\
\left\{\varphi(\alpha_t): t\in\mathcal{I}_0\right\}$ to  construct a finite chain from $\varphi(z_1)$ to $\varphi(z_2)$. Because of this, the next inequality
\begin{equation}\label{inequality-b}
|\varphi(z_1)-\varphi(z_2)|<\sum_{j\in\mathcal{J}_0}{\rm diam}(\Gamma_j)+\sum_{k\in\mathcal{K}_0}{\rm diam}(V_k)+\sum_{t\in\mathcal{I}_0}{\rm diam}(\varphi(\alpha_t))< \frac12\epsilon+\frac14\epsilon+\frac14\epsilon=
\epsilon
\end{equation}
is also true. By the flexibility of $z_1,z_2\in C_r(z_0)\cap D$, we have completed our proof.
\end{proof}

Now, combing the results of Theorems \ref{oscillation-a} and \ref{topological-cct}, we readily have Theorem \ref{arsove-2}. Moreover,
we may assume that the circle domain $D$ in {\bf Theorem \ref{arsove-2}} is just a generalized Jordan domain and keep the other assumptions. In other words, we shall have the following theorem.

\begin{theo}\label{arsove-new}
Let $\Omega_1$ be a generalized Jordan domain and $\varphi: \Omega_1\rightarrow \Omega_2\subset\hat{\mathbb{C}}$  a conformal homeomorphism. Suppose that the point components of $\partial\Omega_2$ form a set with $\sigma$-finite linear measure and $\partial\Omega_2$ has at most countably many non-degenerate components $\{P_n\}$ with $\displaystyle\sum_n{\rm diam}(P_n)<\infty$.  Then $\varphi$  extends continuously to $\overline{\Omega_1}$ if and only if $\partial\Omega_2$ is a Peano compactum.
\end{theo}
\begin{proof}The proof is similar to and easier than that of Theorem \ref{arsove_sigma-new}, so we omit it here.
\end{proof}

\subsection{Generalization of Osgood-Taylor-Carath\'eodory Theorem}\label{final}
Theorem \ref{OTC-b} concerns generalized Jordan domains that may not be cofat or countably connected. The proof is given below.

\begin{proof}[{\bf Proof for Theorem \ref{OTC-b}}]
If the point components  of $\partial\Omega_2$ form a set of $\sigma$-finite linear measure  we apply Theorem \ref{arsove-new} to the map $\varphi: \Omega_1\rightarrow\Omega_2$ and obtain a well-defined continuous extension $\overline{\varphi}: \overline{\Omega_1}\rightarrow\overline{\Omega_2}$.  Then, applying Theorem \ref{arsove_sigma-new}, we see that the inverse map $\psi=\varphi^{-1}: \Omega_2\rightarrow \Omega_1$ also extends to be a continuous map $\overline{\psi}: \overline{\Omega_2}\rightarrow\overline{\Omega_1}$. Consequently, we can check that $\overline{\varphi}\circ \overline{\psi}=id_{\overline{\Omega_2}}$ and $\overline{\psi}\circ \overline{\varphi}=id_{\overline{\Omega_1}}$. This indicates that $\overline{\varphi}$ and $\overline{\psi}$ are both injective.

If the point components of $\partial \Omega_1$ form a set of $\sigma$-finite linear measure  we apply Theorem \ref{arsove_sigma-new} to the map $\varphi: \Omega_1\rightarrow\Omega_2$ and obtain a well-defined continuous extension $\overline{\varphi}: \overline{\Omega_1}\rightarrow\overline{\Omega_2}$.  Then, applying Theorem \ref{arsove-new} to the inverse map $\psi=\varphi^{-1}: \Omega_2\rightarrow \Omega_1$, we obtain another continuous map $\overline{\psi}: \overline{\Omega_2}\rightarrow\overline{\Omega_1}$ that extends $\psi$. Similarly, we can infer that $\overline{\varphi}$ and $\overline{\psi}$ are both injective.
\end{proof}

\section{Examples Based on Transboundary Extremal Length}\label{Example}

We shall construct a concrete example to show the failure of Theorems \ref{arsove-1} and \ref{arsove-2} without {\bf Assumption 2}. To do that, we need to recall some basics from the theory of transboundary extremal length introduced in \cite{Schramm95}.

Given a domain $U\subset\widehat{\mathbb{C}}$, let $\mathcal{D}_U$ be the upper semi-continuous decomposition of $\widehat{\mathbb{C}}$ consisting of singletons $\{z\}$ with $z\in U$ and the components of $\widehat{\mathbb{C}}\setminus U$. By Moore's decomposition theorem \cite{Moore25}, the resulting hyperspace is topologically equivalent to the sphere. Let $\pi_U: \widehat{\mathbb{C}}\rightarrow\mathcal{D}_U$ be the natural projection, which is often denoted as $\pi$ when the underlying domain $U$ is clearly understood.

Following Schramm \cite{Schramm95}, we denote by $\mathcal{E}(U)$ the above hyperspace $\mathcal{D}_U$ and call it the {\bf ends compactification} of $U$. Moreover, for any $p\in\mathcal{E}(U)$ or $p\subset\mathcal{E}(U)$, we use $\lceil p\rceil$ to denote the pre-image $\pi_U^{-1}(p)$. We also denote by $\mathcal{C}(U)$ the collection of all those $\pi_U(Q)$ with $Q$ running through the components of $\widehat{\mathbb{C}}\setminus U$. In particular, $\mathcal{C}_i(U)$ consists of those $p\in\mathcal{C}(U)$ with $\lceil p\rceil$ being a singleton and $\mathcal{C}_s(U)=\mathcal{C}(U)\setminus\mathcal{C}_i(U)$.

Let $\mu_0$ be a measure  defined by $\mu_0(B)=Area(U\cap\pi_U^{-1}(B))$ for any Borel set $B\subset\mathcal{E}(U)$, using the spherical area. Clearly, $\mu_0(\mathcal{C}(U))=0$. Let $\mu_1$ be the counting measure on $\mathcal{C}(U)$. Then $\mu=\mu_0+\mu_1$ is a measure on the Borel sigma-algebra of $\mathcal{C}(U)$, which will be considered as a copy of the sphere $\widehat{\mathbb{C}}$.

Given an extended metric $m$,  which means a nonnegative function $m: \mathcal{E}(U)\rightarrow[0,\infty]$, its area is defined as $\displaystyle A(m)=\int m^2(x)d\mu$.
The $m$-length of a curve $\gamma: I\rightarrow\mathcal{E}(U)$, denoted $L_m(\gamma)$, consists of two parts. One is the integral
\[
\displaystyle \int_{\gamma^{-1}(U)}m(\gamma(t))|d\gamma(t)|.
\]
Note that $U$ may be considered as a subset of $\mathcal{E}(U)$, so that $\gamma^{-1}(U)$ makes sense.
The other is the summation
$\displaystyle\sum_{p\in\gamma(I)\setminus U}m(p)$. The $m$-length of a collection $\Gamma$ of curves in $\mathcal{E}(U)$ is defined to be $L_m(\Gamma)=\inf_{\gamma\in\Gamma}L_m(\gamma)$. The {\bf (transboundary) extremal length} of $\Gamma$ is given by
\[
\displaystyle EL(\Gamma)=EL_U(\Gamma)=\sup_m\frac{L_m^2(\Gamma)}{A(m)},
\]
where the supremum is taken over all the extended metrics $m$ on $\mathcal{E}(U)$ that have finite area.

The next two propositions are useful in our construction.

\begin{prop}[{\bf\cite[Lemma 1.1]{Schramm95}}]\label{invariance}
Let $f: U\rightarrow U^*$ be a conformal homeomorphism between domains in $\widehat{\mathbb{C}}$, and let $\Gamma$ be a collection of curves in $\mathcal{E}(U)$. Set $\Gamma^*=\{f\circ\gamma: \gamma\in\Gamma\}$. Then $EL_U(\Gamma)=EL_{U^*}(\Gamma^*)$.
\end{prop}

\begin{prop}[{\bf \cite[Theorem 6.1]{Schramm95}}]\label{geom_ineq}
Let $U\subset\widehat{\mathbb{C}}$ be a $\tau$-cofat domain, $0<\tau<1$, and let $A_1,A_2\subset\mathcal{E}(U)$ be two connected disjoint sets. Let $\Gamma$ be the collection of all curves $\gamma$ in $\mathcal{E}(U)$ that have one endpoint in $A_1$ and the other endpoint in $A_2$, let $d$ be the distance from $\lceil A_1\rceil$ to $\lceil A_2\rceil$, and set
$a=\min\{{\rm diameter}\left(\lceil A_1\rceil\right), {\rm diameter}\left(\lceil A_2\rceil\right)\}$. Then
\begin{enumerate}
\item $EL(\Gamma)\le 104\tau^{-1}(1+\frac{d}{a})$.
\item Assuming that $\mathcal{C}(U)$ is countable,
$\displaystyle EL(\Gamma)\ge\frac{\tau d^2}{13(a+d)^2}$.
\end{enumerate}
\end{prop}

\begin{exam}\label{Bishop_Example}
For any $n\ge1$, let $A_n=\left\{z\in\mathbb{C}: 1+2^{-n-1}\le|z|\le1+2^{-n}\right\}$. Consider $A_n$ as the range of $\phi_n: [0,2\pi]\times[0,1]$, with
$\displaystyle \phi(u,v)=\left(1+2^{-n-1}(1+v)\right)\exp(\mathbf{i}u)$.
Let $j_n\ge1$ be the only integer with $\pi\sqrt{n}-1\le j_n<\pi\sqrt{n}$.
Choose $2j_n$ horizontal segments $\alpha_{n,t}(1\le t\le 2j_n)$ of length $\frac{2}{\sqrt{n}}$ in $[0,2\pi]\times[0,1]$.
\begin{figure}[ht]
\begin{center}
\begin{tikzpicture}[x=1.382cm,y=2.5cm,scale=1.618]

\draw[gray,ultra thick] (0,0)--(6.28,0) --(6.28,1.5) --(0,1.5)--(0,0);

\draw[gray,thick,style=dashed] (0,1)--(6.28,1);
\fill[gray!61.8,thin] (5.28,0.95)--(5.61,0.95)--(5.61,0.90)--(5.28,0.90)--(5.28,0.95);

\foreach \i in {0,...,8}
{
    \draw[black,very thick] (\i*1/3,1-\i*0.05) -- (2/3+\i*1/3,1-0.05*\i);
    \draw[blue,very thick] (6.28-\i*1/3,1-\i*0.05) -- (6.28-2/3-\i*1/3,1-0.05*\i);
}

\draw[gray, thin] (8/3,-0.1)--(8/3,1.7);
\fill[black] (8/3,1.65) node[left]{$u=\frac{j_n-1}{\sqrt{n}}$};
\draw[gray,thin] (3.14,-0.1)--(3.14,1.7);
\fill[black] (3.14,1.65) node[right]{$u=\pi$};

\draw[gray, thin] (6.28-2/3,-0.1)--(6.28-2/3,1.7);
\fill[black] (6.28-2/3,1.65) node[right]{$u=2\pi\!-\!\frac{2}{\sqrt{n}}$};
\draw[gray,thin] (6.28-1,-0.1)--(6.28-1,1.7);
\fill[black] (5.28,1.65) node[left]{$u=2\pi\!-\!\frac{3}{\sqrt{n}}$};

\end{tikzpicture}
\end{center}
\vskip -0.5cm
\caption{The segments $\alpha_{n,t}$ and how they are distributed, with $n=9=j_n$.}\label{small_arcs}
\end{figure}
Arrange these segments symmetrically about the vertical line $u=\pi$, as displayed in Figure \ref{small_arcs}, such that they satisfy each of the following  conditions.
\begin{itemize}
\item $\phi_n\left(\alpha_{n,1}\right)\cap\phi_n\left(\alpha_{n,2j_n}\right)$ is a singleton.
\item  $\phi_n\left(\alpha_{n,j_n}\right)\cap\phi_n\left(\alpha_{n,j_n+1}\right)\ne\emptyset$.
\item  $\alpha_{n,t+1}$ is to the right of and slightly {\bf lower} than $\alpha_{n,t}$ for $1\le t\le j_n-1$.
\item  $\alpha_{n,j_n+t+1}$ is to the right of and slightly {\bf higher} than $\alpha_{n,j_n+t}$ for $1\le t\le j_n-1$.
\end{itemize}
Let $\Omega\subset\mathbb{C}$ be the domain obtained by removing from $\{z: 1<|z|<2\}$ all those arcs $\beta_{n,t}$ for $n\ge1$ and $1\le t\le 2j_n-2$, defined by
\begin{equation}\label{beta_n_t}
\beta_{n,t}=\left\{\begin{array}{ll}
\phi_n(\alpha_{n,1}\cup\alpha_{n,2j_n})& t=1\\
\phi_n(\alpha_{n,t})& 2\le t\le j_n-1\\
\phi_n(\alpha_{n,j_n}\cup\alpha_{n,j_n+1})& t=j_n\\
\phi_n(\alpha_{n,t+1})& j_n+1\le t\le j_n-2
\end{array}\right.
\end{equation}
The boundary of $\Omega$ is a Peano compactum that consists of two circles, $C_1=\{z: |z|=1\}$ and $C_2=\{z: |z|=2\}$, together with countably many arcs whose diameters have an {\bf infinite sum} and form a null sequence. By \cite[Theorem 0.1]{He-Schramm93}, there are conformal homeomorphisms $\varphi: D\rightarrow\Omega$ of a circle domain $D$ onto $\Omega$. For any of those $\varphi$, the boundary map $\varphi^B$ sends a point component of $\partial D$ to $C_1$,  a non-degenerate component of $\Omega$. Let $P_0$ be this point component of $\partial D$. Then $\varphi$ can not be continuously extended over $P_0$.
\end{exam}

In the sequel, the domain $\Omega$ is given as in Example \ref{Bishop_Example}. We will prove three lemmas that provide the technical details for Example \ref{Bishop_Example}.

\begin{lemm}\label{estimating_area}
Let the extended metric $m: \mathcal{E}(\Omega)\rightarrow[0,\infty]$ be defined by
\begin{equation}\label{extended_metric}
m(z)=\left\{\begin{array}{ll} 1&z\in \Omega\cup\pi(C_1)\cup\pi(C_2) \\ n^{-0.8}& z=\pi(\beta_{n,t})\  \text{for\ some}\ n\ge1 \ \text{and\ some}\  1\le t\le 2j_n-2
\end{array}\right.
\end{equation}
Then  $A(m)\in(0,\infty)$.
\end{lemm}
\begin{proof}
By the choice of $j_n$ and $m$, it is routine to check for all $n\ge1$ the following equation
\begin{equation}\label{j_n}
\int_{\left\{\pi(\beta_{n,t}): \ 1\le j\le 2j_n-2\right\}}m(z)^2d\mu=\frac{2j_n-2}{n^{1.6}}<\frac{2\pi\sqrt{n}-2}{n^{1.6}}<\frac{2\pi}{n^{1.1}}.
\end{equation}
This ends our proof.
\end{proof}

\begin{lemm}\label{estimating_EL}
Let $m$ be given as in Lemma \ref{estimating_area}. Let $\Gamma$ be the collection of all the curves $\gamma\subset\mathcal{E}(\Omega)$ that connects $\pi(C_2)$  to $\pi(C_1)$. Then $L_m(\Gamma)=\infty$, indicating that $EL(\Gamma)=\infty$.
\end{lemm}
\begin{proof}
Fix a  constant $C>0$ such that the arc length of any curve $\eta\subset\Omega$ as a curve in $\widehat{\mathbb{C}}$ equipped with the spherical metric, which coincides with $L_m(\eta)$, is greater than $C$ times the arc length of $\eta$ as a curve in $\mathbb{C}$ equipped with the Euclidean metric. Also note that every arc $\gamma\in\Gamma$ has to cross the annulus $A_n$ at least once. By the Cut Wire Theorem \cite[p.  72, Theorem 5.2]{Nadler92}, we see that  $\gamma\cap A_n$ has at least one component, denoted as $\gamma_n$, that intersects the two boundary components of $A_n$ both. Then there are two possibilities: either $\gamma_n\ni\pi(\beta_{n,t})$ for some $t$ or $\gamma_n\subset\Omega$. In the first, we have $L_m(\gamma_n)>n^{-0.8}$. In the second,
the choices of the circular slits $\beta_{n,t}$ indicate that $\phi^{-1}(\gamma_n)$ crosses a rectangle $R_n$
whose width is no less than $\frac{1}{\sqrt{n}}$, which further indicates that the arc length of $\gamma_n$ as a curve in $\mathbb{C}$ is greater than $\frac{1}{\sqrt{n}}$. See the right part of Figure \ref{small_arcs} for such a rectangle, depicted in shadow and located between two of the horizontal segments  $\alpha_{n,t}$. The left and right sides lie on two vertical lines, whose equations are both marked.
By the choice of $C>0$, we have $L_m(\gamma_n)\ge\frac{C}{\sqrt{n}}$.
Therefore, we can conclude that $L_m(\gamma_n)>\frac1n$ holds for all but finitely many $n\ge1$. This indicates that the $m$-length of $\gamma$, denoted as $L_m(\gamma)$, is infinite. By the flexibility of $\gamma\in\Gamma$, we further see that the $m$-length  of the whole collection $\Gamma$, defined by $L_m(\Gamma)=\inf_{\gamma\in\Gamma}L_m(\gamma)$, is infinite. Consequently, the extremal length of $\Gamma$ is also infinite.
\end{proof}

\begin{lemm}\label{CT_failed}
Let $\Omega$ be given as in Example \ref{Bishop_Example}. Let $\Gamma$ be given as in Lemma \ref{estimating_EL}. Let $\varphi: D\rightarrow\Omega$ be a conformal homeomorphism of a circle domain $D$ onto $\Omega$. Then the boundary map $\varphi^B$ sends a point component of $\partial D$ to $C_1$.
\end{lemm}
\begin{proof} Let $\psi=\varphi^{-1}$.
Let $\pi_D: \widehat{\mathbb{C}}\rightarrow\mathcal{E}(D)$ be the natural projection. Let $\Gamma^*=\{\psi\circ\gamma: \gamma\in\Gamma\}$ be the collection  consisting of all curves $\psi(\gamma)$ with $\gamma\in\Gamma$. Note that $\psi$ gives rise to a homeomorphism from $\mathcal{E}(\Omega)$ onto $\mathcal{E}(D)$ that extends $\psi$ itself. From this we can infer that $\Gamma^*$ coincides with the collection of all curves in $\mathcal{E}(D)$ connecting $\pi_D\left(\psi^B(C_2)\right)$ to $\pi_D\left(\psi^B(C_1)\right)$. Now, applying Proposition \ref{invariance} we see that the extremal length of $\Gamma^*$ is also infinite. Applying Proposition \ref{geom_ineq}, we further see that one of the two boundary components, $\psi^B(C_2)$ and $\psi^B(C_1)$, is a singleton. As $\psi^B(C_2)$ is isolated, we see that  $\psi^B(C_1)$ is a singleton, which corresponds to $C_1$ under  $\varphi^B$.
\end{proof}

To conclude this section, we present the next remark.
\begin{rema}\label{Bishop_Example-b}
In Example \ref{Bishop_Example} we construct a concrete domain $\Omega$ that satisfies all conditions of {\bf Theorem \ref{arsove-1}} and {\bf \ref{arsove-2}}, except for {\bf Assumption 2}. Since $\partial\Omega$ has countably many components, there are conformal homeomorphisms from a circle domain $D$ onto $\Omega$. However, none of those homeomorphisms extends continuously to the closure of $D$. Moreover, we can change  $\Omega$ appropriately so that it becomes a generalized Jordan domain $\Omega_1$, by slightly thickening the arcs $\beta_{n,t}\subset A_n$ to be pairwise disjoint Jordan curves. Since $\partial\Omega_1$ has countably many components, there exist conformal homeomorphisms from $\Omega_1$ to some circle domain $D$. On the one hand, every conformal homeomorphism $h: \Omega_1\rightarrow D$  has a continuous extension $\overline{h}: \overline{\Omega_1}\rightarrow\overline{D}$. On the other, the inequality $\displaystyle \sum_n{\rm diam}(P_n)+\sum_n{\rm diam}(Q_n)<\infty$ in {\bf Theorem \ref{OTC-b}} does not hold and the extension $\overline{h}$ is one-to-one everywhere except at points on $C_1=\partial\mathbb{D}$, which is sent to a point component of $\partial D$. That is to say, the assumption $\displaystyle \sum_n{\rm diam}(P_n)+\sum_n{\rm diam}(Q_n)<\infty$ in {\bf Theorem \ref{OTC-b}} can not be removed.
\end{rema}

\noindent
{\bf Acknowledgement}. The authors are grateful to Christopher Bishop (Stony Brook of SUNY),  to Malik Younsi (University of Hawaii at Manoa), to Xiaoguang Wang (Zhejiang University) and to Weiyuan Qiu (Fudan University) for  helpful suggestions and/or discussions. In particular, the construction of Example \ref{Bishop_Example} is essentially based on private communication in emails between Christopher Bishop and one of the authors. The authors also owe their thanks to anonymous referee(s) for many suggestions that are of great help for the revisions, especially those related to the updated form of Theorem \ref{OTC-b} and the newly added Lemma \ref{Wolff_lemma_2}, which is directly used in our proofs. Finally, the first named author is supported by
Chinese National Natural Science Foundation Projects
\# 11871483 and \#11771391.

\bibliographystyle{plain}

\end{document}